\documentclass[12pt]{article}
 
\usepackage{amstext,amssymb,amsmath,stmaryrd,amsthm}
\usepackage{geometry} 
\usepackage{color}
\usepackage[nottoc]{tocbibind} 
\usepackage{hyperref}
\usepackage{makeidx}
\makeindex

\numberwithin{equation}{section}

\newtheorem{theorem}{Theorem}[section]
\newtheorem{corollary}[theorem]{Corollary}
\newtheorem{lemma}[theorem]{Lemma}
\newtheorem{proposition}[theorem]{Proposition}

\newtheorem{rem}[theorem]{Remark}
\newtheorem{Example}[theorem]{Example}

\newcommand{\U}{\mathcal{U}}
\newcommand{\Ent}{\textbf{Ent}}
\newcommand{\G}{\mathbb{G}}

\newcommand{\R}{\mathbb{R}}

\newcommand{\Z}{\mathbb{Z}}
\newcommand{\E}{\mathbb{E}}

\newcommand{\Ind}{\textbf{1}}

\date{ }

\begin{document}

\label{begindoc} 
\title{\Large{From $U$-bounds to Isoperimetry} \\ \large{with applications to H-type groups}
\thanks{
{Supported by EPSRC 
EP/D05379X/1}} 
 }

\author{ J. Inglis , V. Kontis , 
B. Zegarli{\'n}ski $^\ddag$  \footnote{
On leave of absence from Imperial College London. 
}  \\
{\small{Imperial College London}}\\
{\small{$^\ddag$CNRS, Toulouse}}\\
}

\maketitle

{\small{
\noindent\textbf{Abstract}: 
{\em In this paper we study applications of $U$-bounds to coercive and isoperimetric
problems for probability measures on finite and infinite products of H-type groups.} 
}}\\

\tableofcontents

\section{Introduction}\label{Introduction}

An effective technology to study coercive inequalities involving (sub-) gradients and a variety of probability measures on metric measure spaces was recently introduced in \cite{H-Z}.   This approach was
based on so-called $U$-bounds, that is estimates of the following form
\[ 
\int |f|^qU(d)^{\gamma_q} d\mu \leq C_q
\int|\nabla f |^q d\mu + D_q \int |f |^q d\mu.
\]
Here $q\in [1,\infty)$, $d$ is a metric associated to the (sub-) gradient $\nabla$, $\gamma_q, C_q, D_q\in(0,\infty)$ are constants independent of the function $f$, and $d\mu \equiv e^{-U(d)} d\lambda$ is a probability measure, where $U(d)$ is a function that is bounded from below and has suitable growth at infinity, and $d\lambda$ is a natural underlying measure.
While the consequences of the bounds corresponding to $q>1$ were extensively explored there, the limiting case was left open. In this paper we show that there is a natural direct way from $U$-bounds with $q=1$ to isoperimetric information. In fact we show an essential equivalence of such a bound with an
$L_1\Phi$-entropy inequality
\[
\mathbf{Ent}_\mu^\Phi( f ) \leq c\ \mu|\nabla f|
\]
where
\[\mathbf{Ent}_\mu^\Phi( f ) \equiv \mu \Phi( f ) - \Phi( \mu f )
\] 
is defined with a suitable Orlicz function $\Phi$, as well as the equivalence with an
isoperimetric inequality with a suitable profile function.
We first recall an  interesting result of \cite{Led1} showing that in case of the Gaussian measures on Euclidean spaces, the functions $f$ such that $\mu|f|<\infty$ belong to the Orlicz space defined by a function $\Phi(s) = s \left(\log(1+s)\right)^\frac12$. Also, on the level of isoperimetry for probability measures, we would like to recall a comprehensive characterisation of isoperimetric profiles for measures on the real line obtained in \cite{bobkovhoudre} (see also \cite{Barthe-Cattiaux-Roberto,B-Z, CGGR, M} and references therein) as well as the isoperimetric functional inequalities studied in \cite{bobkovifi}, (\cite{BBBC, BL, B-Z, Z}).  These results provided additional motivation to our work. In particular, in \cite{B-Z}
the authors conjecture that for super-Gaussian distributions one should expect an analog of the isoperimetric functional inequality ($IFI_2$) introduced in \cite{bobkovifi}, with a suitable non-Gaussian isoperimetric function and a different than Euclidean length of the gradient. In \cite{BBBC} (an alternative to \cite{HQL}) the authors gave a proof of the $p=1$ (sub-) gradient bound 
\[|\nabla P_tf|^p\leq C_p(t) P_t |\nabla f|^p\]
for the heat kernel on the Heisenberg group, and as a consequence obtained an $IFI_2$ inequality in this case.  We mention that, for $p>1,$ gradient bounds were earlier established in \cite{DM}, while the logarithmic Sobolev inequality for heat kernels on Heisenberg-type groups was established in \cite{H-Z}.   

The other interesting question is what are the optimal equivalent conditions, on the one side characterising the properties of the semigroup
for which the form associated to the generator is given by the square of a fixed sub-gradient, and  
on the other side characterising the isoperimetric properties (e.g. in the form of some isoperimetric functional inequality with a given length of the
sub-gradient). In the particular situation when $p=1$ gradient bounds are known,
and an equivalence relation (between $IFI_2$ and the logarithmic Sobolev inequality) was established in \cite{F}. It seems that we are still away from fully understanding the peculiarity of this situation and in particular answering the question what kind of additional conditions are necessary to establish equivalence between conditions of different orders in the length of the gradient (as well as finding a more direct proof of this equivalence without going through the semigroup route). 

From the point of view of applications to an infinite dimensional probabilistic setup involving an infinite product of non-compact Lie groups, it is important that we are dealing with inequalities satisfying
the tensorisation property.  Then one can attack the interesting question of for which non-product measures
one can prove similar properties. This question, when the underlying space is as we wish,
appears to have some new challenging features and so far, besides the results of \cite{I-P} where logarithmic Sobolev inequalities $LS_q$, $q>1$, are shown for some classes of measures, not much is known.
Therefore in the present paper we also contribute to this topic by proving tight $L_1\Phi$-entropy inequalities for suitable infinite dimensional Gibbs measures. 

The organisation of our paper is as follows. In section 2, starting from $U$-bounds, we prove the $L_1\Phi$-entropy inequality via a route involving ``dressing up'' the classical Sobolev inequality
and a tightening procedure using a generalised Rothaus type lemma of \cite{L-Z}, extended relative entropy
bounds of \cite{F-R-Z}, and the following Cheeger type inequality 
\[ 
\mu |f-\mu f| \leq c_0 \mu |\nabla f|.
\]
In fact, this type of Cheeger inequality is shown (in Theorem \ref{cheeger from U-bound}) to be a simple consequence of a similar inequality in balls together with
$U$-bounds, provided the function $U$ grows to infinity with the size of the ball.

In section 3 we discuss some applications to isoperimetric and functional isoperimetric inequalities.  Section 4 contains some consequences of the $L_1\Phi$ -entropy inequality. In particular this includes the $LS_q$ inequality and $U$-bounds.  In Theorem \ref{summary} we summarise all interrelations between the properties discussed before. 
Section 5 is devoted to applications of the theory developed in the previous sections to the important
class of H-type groups, where one can check the $U$-bounds for probability measures with density (essentially) dependent on the Carnot-Carath\'eodory distance. The interesting outcome, which comes out naturally within the presented approach, includes a proof of the $p=1$ subgradient bounds for heat kernels
on H-type groups which could  potentially be extended to more complicated non-compact groups.
Finally in section 6 we prove the $L_1\Phi$-entropy inequalities for non-product probability measures 
on an infinite product of H-type groups, which allows us in particular to obtain some new isoperimetric information. Additionally we prove here the $IFI_2$ inequality in such a setup; in fact even when we are using the full gradient, this provides an interesting extension of results in \cite{Z} allowing us to include the important case of unbounded interactions.

\bigskip

\section{$L_1\Phi$-entropy inequalities from $U$-bounds}
\label{L_1 entropy inequalities from U-bounds} 
Throughout this paper we will be working in $\mathbb{R}^N$ equipped with a metric $d: \R^N\times \R^N \to [0, \infty)$ and Lebesgue measure $d\lambda$.  For $r\geq0$, we will set 
\[
B(r) := \{x:d(x)\leq r\},
\]
 where $d(x) := d(x,0)$.
 
We will also let $\nabla$ be a general sub-gradient in $\mathbb{R}^N$ i.e. $\nabla$ is a finite collection $\{X_1, \dots, X_m\}$ of possibly non-commuting fields.  Assume that  the divergence of each of these fields with respect to the Lebesgue measure $\lambda$ on $\mathbb{R}^N$ is zero.  Set $\Delta := \sum_{i=1}^m X^2_i$ and $|\nabla f| = \left(\sum_{i=1}^m (X_if)^2\right)^\frac{1}{2}$. 

\begin{theorem}\label{thm1.1}
\label{defective beta LS}
Let $U$ be a locally Lipschitz function on $\R^N$, which is bounded from below and is such that $Z=\int e^{-U}d\lambda <\infty$.  Let $d\mu = \frac{e^{-U}}{Z}d\lambda$, so that $\mu$ is a probability measure on $\R^N$.  Suppose that the following classical Sobolev inequality is satisfied
\begin{equation}
\label{class sob}
\left(\int |f|^{1+\varepsilon}d\lambda\right)^\frac{1}{1+\varepsilon} \leq a\int |\nabla f|d\lambda + b\int |f|d\lambda
\end{equation}
for some constants $a,b\in[0, \infty)$ and $\varepsilon>0$, and that for some $A, B \in [0, \infty)$ we have
\begin{equation} \label{U-bound}
\mu\left(|f|\left(|U|^\beta + |\nabla U|\right)\right)\leq A\mu|\nabla f| + B\mu|f|
\end{equation}
for some $\beta\in(0,1]$.  Then there exist constants $C, D\in[0, \infty)$ such that
\begin{equation}
\label{defective}
\mu\left(|f|\left|\log\frac{|f|}{\mu|f|}\right|^\beta\right) \leq C\mu|\nabla f| + D\mu|f|.
\end{equation}
\end{theorem}

\begin{proof}
Without loss of generality, we may suppose that $f\geq0$ and $U\geq0$.  Indeed, otherwise we may apply \eqref{defective} to the positive and negative parts of $f$ separately.  Moreover, if  $U\geq -K$, with $K\geq0$, we have that $U + K \geq 0$ and then we can replace $f$ by $f e^{-K}$ in \eqref{defective}.  

First note that
\begin{align*}
\mu\left(f\left|\log \frac{f}{\mu f}\right|^\beta\right) &= \mu\left( f\left[\log_+\frac{f}{\mu f}\right]^\beta\right) + \mu\left(f\left[\log\frac{\mu f}{f}\right]^\beta\Ind_{\{f\leq\mu f\}}\right) \\
&\leq \mu\left( f\left[\log_+\frac{f}{\mu f}\right]^\beta\right) + e^{-\beta}\beta^\beta \mu(f),
\end{align*}
since $\sup_{x\in(0,1)}x\left(\log\frac{1}{x}\right)^\beta = e^{-\beta}\beta^\beta$.  Thus it suffices to prove that
\begin{equation}
\label{defective1}
\mu\left(f\left[\log_+\frac{f}{\mu f}\right]^\beta\right) \leq C\mu|\nabla f| + D \mu (f).
\end{equation}
with some constants $C, D\in(0,\infty)$ independent of $f$.
Suppose that $\mu(f) =1$.  With $F\equiv fe^{-U}$ and $\varepsilon\in(0,1)$ sufficiently small, we have
\begin{equation}
\label{RHS}
\int F\left[\log_+\left(F\right)\right]^\beta d\lambda  = \int_{\{F \geq 1\}}F \left[\frac{1}{\varepsilon}\log\left(F\right)^\varepsilon\right]^\beta d\lambda.
\end{equation}
Now, by Jensen's inequality (since, for $\beta\in(0,1]$, the function $\left(\log x\right)^\beta$ is concave on $x\geq1$)
\begin{align*}
 \int_{\{F \geq 1\}} F\left[\frac{1}{\varepsilon}\log\left(F\right)^\varepsilon\right]^\beta d\lambda  
& = \frac{\int_{\{F \geq 1\}}Fd\lambda}{\varepsilon^\beta}\int_{\{F \geq 1\}}\frac{F}{\int_{\{F \geq 1\}}F d\lambda}\left[\log\left(F\right)^\varepsilon\right]^\beta d\lambda\\
& \leq \frac{\int_{\{F \geq 1\}}Fd\lambda}{\varepsilon^\beta}\left[\log\frac{ \int_{\{F \geq 1\}}\left(F\right)^{1+\varepsilon}d\lambda}{\int_{\{F \geq 1\}}F d\lambda}\right]^\beta\\
&= \frac{(1+\varepsilon)^\beta\int_{\{F \geq 1\}}Fd\lambda}{\varepsilon^\beta}\left[\log\frac{ \left(\int_{\{F \geq 1\}}\left(F\right)^{1+\varepsilon}d\lambda\right)^\frac{1}{1+\varepsilon}}{\left(\int_{\{F \geq 1\}}F d\lambda\right)^\frac{1}{1+\varepsilon}}\right]^\beta\\
& \leq  \frac{(1+\varepsilon)^\beta\int_{\{F \geq 1\}}Fd\lambda}{\varepsilon^\beta}\left[\log\frac{ \left(\int_{\{F \geq 1\}}\left(F\right)^{1+\varepsilon}d\lambda\right)^\frac{1}{1+\varepsilon}}{\left(\int_{\{F \geq 1\}}F d\lambda\right)^\frac{1}{1+\varepsilon}} + 1\right],
\end{align*}
using the simple fact that $x^\beta \leq x + 1$ for all $x\geq0$.  Thus, since $\log x \leq x-1$ for all $x\geq0$,
\begin{align*}
\int_{\{F \geq 1\}}F\left[\frac{1}{\varepsilon}\log\left(F\right)^\varepsilon\right]^\beta d\lambda  
\leq \frac{(1+\varepsilon)^\beta\int_{\{F \geq 1\}}Fd\lambda}{\varepsilon^\beta}\left[\frac{ \left(\int_{\{F \geq 1\}}\left(F\right)^{1+\varepsilon}d\lambda\right)^\frac{1}{1+\varepsilon}}{\left(\int_{\{F \geq 1\}}F d\lambda\right)^\frac{1}{1+\varepsilon}}  \right].
\end{align*}
Since we have assumed that $\mu(f) =1$, we have $\int_{\{fe^{-U} \geq 1\}}fe^{-U}d\lambda/Z\equiv \int_{\{F \geq 1\}}Fd\lambda/Z \leq 1$, and so
\begin{align*}
\frac{1}{\left(\int_{\{F \geq 1\}}F d\lambda\right)^\frac{1}{1+\varepsilon}} &\leq \frac{Z^\frac{\varepsilon}{1+\varepsilon}}{\int_{\{F \geq 1\}}F d\lambda}.
\end{align*}
Thus
\begin{align*}
\int_{\{F \geq 1\}} F\left[\frac{1}{\varepsilon}\log\left(F\right)^\varepsilon\right]^\beta d\lambda 
& \leq\frac{(1+\varepsilon)^\beta\int_{\{F \geq 1\}}Fd\lambda}{\varepsilon^\beta}\left[Z^\frac{\varepsilon}{1+\varepsilon}\frac{ \left(\int_{\{F \geq 1\}}\left(F\right)^{1+\varepsilon}d\lambda\right)^\frac{1}{1+\varepsilon}}{\int_{\{F \geq 1\}}F d\lambda}  \right]\\
&= \frac{(1+\varepsilon)^\beta Z^\frac{\varepsilon}{1+\varepsilon}}{\varepsilon^\beta} \left(\int \left(F\right)^{1+\varepsilon}d\lambda\right)^\frac{1}{1+\varepsilon} \\ 
& \leq \frac{(1+\varepsilon)^\beta Z^\frac{\varepsilon}{1+\varepsilon}}{\varepsilon^\beta}\left(a \int |\nabla (F)| d\lambda + bZ\right), \\
\end{align*}
provided $\varepsilon>0$ is  chosen sufficiently small so that in the last step we can apply the classical Sobolev inequality \eqref{class sob}.  Dividing both sides by the normalisation factor $Z$ and recalling $F\equiv f e^{-U}$,  this implies
\begin{equation}
\label{RHS2}
\int f \left[\log_+\left(f e^{-U}\right)\right]^\beta d\mu \leq c_1 \mu |\nabla f| + c_2\mu(f|\nabla U|) + c_3,
\end{equation}
with $d\mu\equiv \frac{1}{Z}e^{-U}d\lambda$ and $c_1 = c_2 =  (1+\varepsilon)^\beta a Z^\frac{\varepsilon}{1+\varepsilon}/\varepsilon^\beta$, $c_3 = (1+\varepsilon)^\beta bZ^\frac{\varepsilon}{1+\varepsilon}/\varepsilon^\beta$. 
Now consider the left-hand side of \eqref{RHS2}.  Since $\beta \in (0,1]$ and $U\geq0$, we have
\begin{align*}
\int f  \left[\log_+\left(f e^{-U}\right)\right]^\beta d\mu &=  \int_{\{f\geq e^U\}}f  \left(\log f - U\right)^{\beta}d\mu \\
&\geq  \int_{\{f \geq e^U\}}f  \left(\log f\right)^\beta d\mu -  \int_{\{f\geq e^U\}}fU^{\beta} d\mu\\
& = \mu\left( f \left[\log_+ f\right]^\beta\right) -  \int_{\{ 1\leq f \leq e^U\}}f  \left(\log f\right)^\beta d\mu 
  -  \int_{\{f\geq e^U\}}f U^{\beta} d\mu\\
&\geq \mu\left( f \left[\log_+ f\right]^\beta\right) -  \int_{\{ 1\leq f  \}} f U^{\beta} d\mu .
\end{align*}
Combining this with \eqref{RHS2} we see that
\begin{align*}
\mu\left( f \left[\log_+ f\right]^\beta\right) 
 &\leq c_1 \mu |\nabla f| + c_2\mu(f|\nabla U|) + c_3  +  \int_{\{ 1\leq f \}}fU^\beta   d\mu \\
 & \leq  c_1 \mu |\nabla f| + \max\{c_2, 1\}\mu\left(f\left(U^\beta + |\nabla U|\right)\right) + c_3 \\
 & \leq \left(c_1 + \max\{c_2, 1\}A\right)\mu |\nabla f| + c_3 + \max\{c_2, 1\}B,
 \end{align*}
where we have used \eqref{U-bound} in the last step.  
Finally, for general $f\geq 0$, we apply the above inequality to $f/\mu(f)$ to arrive at \eqref{defective1}.
 \end{proof}
  
 As a corollary, we can also state the following perturbation result.
  
 \begin{corollary}\label{cor1.2}
 \label{perturbation1}
 Let $U$ and $\mu$ be as in Theorem \ref{defective beta LS}, and suppose conditions \eqref{class sob} and \eqref{U-bound} are satisfied.  Let $W$ be a locally Lipschitz function such that $\int e^{-W}d\mu<\infty$ and
\begin{equation}
\label{potential condition}
|\nabla W| \leq  \delta\left(|U|^\beta + |\nabla U|\right)  + C(\delta), \qquad |W|^\beta \leq a_0 \left(|U|^\beta + |\nabla U|\right) + a_1
\end{equation}
almost everywhere, with some $0<\delta < \frac{1}{A}$ and $C(\delta), a_0, a_1 \in(0, \infty)$.
Then there exist constants $\tilde{C}$ and $\tilde{D}$ such that
\begin{equation}
\tilde\mu\left(|f|\left|\log\frac{|f|}{\tilde\mu|f|}\right|^\beta\right) \leq \tilde{C}\tilde\mu|\nabla f| + \tilde{D}\tilde\mu|f|,
\end{equation}
where $\tilde\mu$ is the probability measure on $\R^N$ given by $\tilde\mu(d\lambda) := e^{-W}\mu(d\lambda)/Z_{\tilde\mu}$, with $Z_{\tilde\mu}\equiv \mu (e^{-W})$.
\end{corollary}

\begin{proof}
Take $f\geq0$.  Since by assumption \eqref{U-bound} holds, we can apply it to the function $f e^{-W}$.  This yields
\begin{align*}
\tilde{\mu}\left(f \left(|U|^\beta + |\nabla U|\right)\right) & \leq A\tilde{\mu}|\nabla f| + A\tilde{\mu}\left(f|\nabla W|\right) + B\tilde{\mu}(f)\\
&\leq A\tilde{\mu}|\nabla f| + \delta A\tilde{\mu}\left(f\left(|U|^\beta + |\nabla U|\right)\right) + (B+AC(\delta))\tilde{\mu}(f)
\end{align*}
using \eqref{potential condition}.  Thus, since $\delta A<1$, we have that
\begin{equation}
\label{perturbed U bound}
\tilde{\mu}\left(f \left(|U|^\beta + |\nabla U|\right)\right) \leq \tilde{A}\tilde{\mu}|\nabla f| +\tilde{B}\tilde{\mu}(f)
\end{equation}
for $\tilde{A} = A/(1-\delta A), \tilde{B} = (B+AC(\delta))/(1-\delta A)$.  Replacing $f$ by $f e^{-W}$ in \eqref{defective}
of   Theorem \ref{defective beta LS}, we get
\[
\tilde{\mu}\left(f\left|\log\frac{f e^{-W}}{\tilde{\mu}(f)Z_{\tilde{\mu}}}\right|^\beta\right) \leq C\tilde{\mu}|\nabla f| + C\tilde{\mu}\left(f|\nabla W|\right) + D \tilde{\mu}(f).
\]
Using this together with \eqref{perturbed U bound}, yields
\begin{align*}
&\tilde{\mu}\left(f\left|\log\frac{f}{\tilde{\mu}(f)}\right|^\beta\right) \leq C\tilde{\mu}|\nabla f| + \tilde{\mu}\left(f\left(|W|^\beta + C|\nabla W|\right)\right) + \left(D + |\log Z_{\tilde{\mu}}|^\beta\right)\tilde{\mu}(f)\\
&\quad \leq C\tilde{\mu}|\nabla f| + a_0\max\{1,C\}\tilde{\mu}\left(f\left(|U|^\beta + |\nabla U|\right)\right) + \left(a_1+D + |\log Z_{\tilde{\mu}}|^\beta\right)\tilde{\mu}(f)\\
& \quad \leq \tilde{C}\tilde{\mu}|\nabla f| + \tilde{D}\tilde{\mu}(f),
\end{align*}
where $\tilde{C} = C + a_0\max\{1, C\}\tilde{A}$ and $\tilde{D} = a_1 + D +  |\log Z_{\tilde{\mu}}|^\beta + a_0\max\{1,C\}\tilde{B}$.  The inequality for general $f$ follows in similar way by applying the above inequality to the positive and negative parts of $f$ separately.
\end{proof}
  
The resulting inequality in Theorem \ref{defective beta LS} is a defective inequality, in the sense that it contains a term involving $\mu|f|$ on the right-hand side.  For our purposes this type of inequality is not strong enough, and therefore we now aim to prove a tightened inequality of the following form
\begin{equation}
\label{beta-LS}
\mathbf{Ent}_\mu^\Phi(|f|):=\mu\left(\Phi(|f|)\right) - \Phi(\mu |f|) \leq c\mu|\nabla f|,
\end{equation}
where $\Phi(x) = x\left(\log(1+x)\right)^\beta, \beta\in(0,1]$, and $c\in(0,\infty)$ is a constant  independent of $f$.  We accomplish this in the situation
(see Theorem \ref{tighten} below) when we have the following Cheeger type inequality  
$$
\mu|f-\mu f|\leq c_0 \mu |\nabla f|
$$
with a constant $c_0\in(0,\infty)$ independent of $f$.  

A bound  of the form described in \eqref{beta-LS} will be called in what follows an \textit{$L_1\Phi$-entropy inequality}.   It is an example of a (non-homogeneous) additive $\Phi$-entropy inequality, as studied in  \cite{Barthe-Cattiaux-Roberto} and \cite{Chafai}.  
To arrive at the desired inequality, our strategy will be as follows. We will first use Theorem  \ref{defective beta LS} to prove a defective $L_1\Phi$-entropy inequality, that is an inequality of a similar form but 
containing additionally on its right-hand side a term proportional to $\mu|f|$.
Then we will adapt some ideas of Rothaus \cite{Ro}, generalised in \cite{B-Z}, to show that such a defective inequality can be tightened.  We begin by proving the following lemma.
 
\begin{lemma}\label{lem1.3}
\label{monotonicity}
Let  $\Phi(x) = x\left(\log(1+x)\right)^\beta, \beta\in(0, 1]$ and let $\mu$ be a given probability measure.   Then there exists a constant $\kappa\in[0, \infty)$ such that for any functions $f$ and $g$ satisfying $0\leq g \leq f$, $\mu f<\infty$, one has
\[
\mathbf{Ent}_\mu^\Phi(g) \leq \mu\left(f \left[\log_+\left(\frac{f}{\mu f}\right)\right]^\beta\right)  + \kappa\mu(f).
\]

\end{lemma}
 
\begin{proof}
We have that
\begin{align}
\label{rel1}
\Ent_\mu^\Phi(g) & = \mu\left(g\left[(\log(1+g))^\beta - (\log(1+\mu g))^\beta\right]\right)\nonumber \\
&\leq \mu\left(g \left[\log\left(1+ \frac{g}{\mu g}\right)\right]^\beta\right)\nonumber\\
&\leq \mu\left(f \left[\log\left(1+ \frac{g}{\mu g}\right)\right]^\beta\right),
\end{align}
since $g\leq f$.
Set $F(x):= (\log(1+x))^\beta$ for $x\in[0, \infty)$.  Then $F$ is increasing and concave.  Moreover, there exists a constant $\theta\in(0,\infty)$ such that $xF'(x) \leq \theta$ for all $x$.  Following \cite{F-R-Z}, we now claim that 
\begin{equation}
\label{claim}
xF(y) \leq xF(x) + \theta y
\end{equation}
for all $x, y \geq 0$.  Indeed, if $y\leq x$ this is trivial.  If $x\leq y$, we have
\begin{align*}
x(F(y) - F(x)) &= x\frac{F(y) - F(x)}{y-x}(y-x)  
 \leq xF'(x)y\\
&\leq \theta y.
\end{align*}
Setting $x = \frac{f}{\mu f}$ and $y=\frac{g}{\mu g}$ in \eqref{claim} and integrating both sides with respect to the measure $\mu$ yields
\[
\mu\left(f \left[\log\left(1+ \frac{g}{\mu g}\right)\right]^\beta\right) \leq \mu\left(f \left[\log\left(1+ \frac{f}{\mu f}\right)\right]^\beta\right) + \theta\mu(f).
\]
Thus, by \eqref{rel1}
\begin{equation}
\label{rel2}
\Ent_\mu^\Phi(g) \leq \mu\left(f \left[\log\left(1+ \frac{f}{\mu f}\right)\right]^\beta\right) + \theta\mu(f).
\end{equation}

Now
\begin{align*}
\mu\left(f \left[\log\left(1+ \frac{f}{\mu f}\right)\right]^\beta\right) &= \mu\left(f \left[\log\left(1+ \frac{f}{\mu f}\right)\right]^\beta\Ind_{\{f\leq \mu f\}}\right) \\
& \qquad + \mu\left(f \left[\log\left(1+ \frac{f}{\mu f}\right)\right]^\beta\Ind_{\{f\geq \mu f\}}\right)\\
&\leq (\log2)^\beta \mu(f) + \mu\left(f \left[\log\left(\frac{2f}{\mu f}\right)\right]^\beta\Ind_{\{f\geq \mu f\}}\right)\\
&=  (\log2)^\beta \mu(f) +  \mu\left(f \left[\log2 + \log\left(\frac{f}{\mu f}\right)\right]^\beta\Ind_{\{f\geq \mu f\}}\right)\\
& \leq 2(\log 2)^\beta\mu(f) + \mu\left(f \left[\log_+\left(\frac{f}{\mu f}\right)\right]^\beta\right),
\end{align*}
using in the last step the elementary inequality $(x+y)^\beta \leq x^\beta + y^\beta$ for $x,y\geq 0$, true when $\beta\in(0,1]$.  Combining this with \eqref{rel2}, we arrive at
\[
\Ent_\mu^\Phi(g) \leq  \mu\left(f \left[\log_+\left(\frac{f}{\mu f}\right)\right]^\beta\right) + \left(2(\log 2)^\beta + \theta\right)\mu(f),
\]
which completes the proof.
\end{proof}

\begin{theorem}\label{thm1.4}
\label{tighten}
Suppose $U$, $\lambda$ and $\mu$ are as in Theorem \ref{defective beta LS}.
In addition, suppose that the following Cheeger type inequality holds
\begin{equation}
\label{cheeger}
\mu|f -\mu f| \leq c_0\mu|\nabla f|
\end{equation}
for some $c_0>0$.  Then there exists $c\in(0,\infty)$ such that \eqref{beta-LS} holds, i.e. for any differentiable function $f$, we have
\[
\mathbf{Ent}_\mu^\Phi(|f|) \leq c\mu|\nabla f|,
\]
where $\Phi(x) = x\left(\log(1+x)\right)^\beta$.
\end{theorem}

\begin{proof}

By Lemma A.1 of the appendix of \cite{L-Z} , we have that there exist constants $\tilde{a}$ and $\tilde{b}$ such that
\[
\Ent_\mu^\Phi(f^2) \leq \tilde{a}\Ent_\mu^\Phi\left((f-\mu f)^2\right) + \tilde{b}\mu(f - \mu f)^2.
\]
Thus, for any $t\in\R$, we have that
\begin{align}
\label{rot1}
\Ent_\mu^\Phi|f+t| &= \Ent_\mu^\Phi\left(|f+t|^\frac{1}{2}\right)^2 \nonumber\\
&\leq \tilde{a}\Ent_\mu^\Phi\left[ \left(|f+t|^\frac{1}{2} - \mu|f+t|^\frac{1}{2}\right)^2\right] + \tilde{b}\mu\left(|f+t|^\frac{1}{2} - \mu|f+t|^\frac{1}{2}\right)^2.
\end{align}
Let $G =\left(|f+t|^\frac{1}{2} - \mu|f+t|^\frac{1}{2}\right)^2$.  Note that we can write
\begin{align*}
G = \left(|f+t|^\frac{1}{2} - \mu|f+t|^\frac{1}{2}\right)^2 & = \left(\int |f(\omega)+t|^\frac{1}{2} - |f\left(\tilde{\omega}\right)+t|^\frac{1}{2}d\mu\left(\tilde{\omega}\right)\right)^2 \\
&\leq \int \left(|f(\omega)+t|^\frac{1}{2} - |f\left(\tilde{\omega}\right)+t|^\frac{1}{2}\right)^2d\mu\left(\tilde{\omega}\right) \\
&\leq \int|f(\omega) - f\left(\tilde{\omega}\right)|d\mu\left(\tilde{\omega}\right)\\
&\leq |f| + \mu |f|,
\end{align*}
using the elementary inequality $\left||x+t|^\frac{1}{2} - |y+t|^\frac{1}{2}\right|\leq |x-y|^\frac{1}{2}$ in the last but one step.  Hence, we have by \eqref{rot1} that
\begin{equation}
\label{G1}
\Ent_\mu^\Phi|f+t| \leq \tilde{a}\Ent_\mu^\Phi\left( G\right) + 2\tilde{b}\mu|f|.
\end{equation}
Since $0\leq G \leq |f| + \mu|f|$, by Lemma \ref{monotonicity} and Theorem \ref{defective beta LS}, we have
\begin{align}
\label{G2}
\Ent_\mu^\Phi\left( G\right) &\leq \mu\left( (|f|+ \mu |f|)\left[\log_+\frac{|f| + \mu |f|}{\mu\left(|f|+ \mu |f|\right)}\right]^\beta\right) + 2\kappa\mu|f|\nonumber\\
&\leq C\mu|\nabla f| + 2(D+\kappa)\mu|f| .
\end{align}
  Combining \eqref{G1} and \eqref{G2} yields
\begin{equation}
\label{G3}
\sup_{t\in\R}\mathbf{Ent}_\mu^\Phi|f+t| \leq \tilde{a}C \mu|\nabla f| + 2(\tilde{a}(D + \kappa) + \tilde{b})\mu|f|.
\end{equation}
This implies the following bound
\begin{equation}
\label{rot2}
\Ent_\mu^\Phi|f| \leq \tilde{a}C\mu|\nabla f| + 2(\tilde{a}(D + \kappa) +\tilde{b})\mu|f-\mu f|.
\end{equation}
Finally we can apply the Cheeger type inequality \eqref{cheeger} to the last term on the right hand side of \eqref{rot2} to arrive at
\[
\Ent_\mu^\Phi\left(|f|\right) \leq c \mu|\nabla f|,
\]
with $c = \tilde{a}C + 2c_0\left(\tilde{a}(D + \kappa) + \tilde{b}\right)$.
\end{proof}

In the same spirit as Corollary \ref{perturbation1}, this inequality is stable under perturbations of the following type.

\begin{corollary}\label{cor1.5}
\label{perturbation2}
Let $U$, $\lambda$ and $\mu$ be as in Theorem \ref{defective beta LS}.
Suppose also that the Cheeger type inequality \eqref{cheeger} holds.  As in Corollary \ref{perturbation1}, let $W$ be a real function which is locally Lipschitz and such that $\int e^{-W}d\mu<\infty$ and
\[
|\nabla W| \leq  \delta\left(|U|^\beta + |\nabla U|\right)  + C(\delta), \qquad |W|^\beta \leq a_0 \left(|U|^\beta + |\nabla U|\right) + a_1
\]
for some $\delta < \frac{1}{A}, C(\delta), a_0, a_1 \in(0, \infty)$ and $\beta\in(0,1]$.  Moreover, let $V$ be a measurable function such that
\[
osc(V) \equiv \sup V - \inf V <\infty.
\]
Then there exists a constant $\hat{c}$ such that
\[
\mathbf{Ent}_{\hat\mu}^\Phi\left(|f|\right)  \leq \hat{c}\hat{\mu}|\nabla f|,
\]
where $\hat\mu$ is the probability measure on $\R^N$ given by 
\[\hat\mu(d\lambda) := e^{-W-V}\mu(d\lambda)/{\hat Z} , 
\]
with a normalisation constant ${\hat Z}\in(0,\infty)$ and  $\Phi(x) = x\left(\log(1+x)\right)^\beta$.
\end{corollary}

\begin{proof}
In the case $V=0$, the result is obtained by following the proof of Theorem \ref{tighten}, using Corollary \ref{perturbation1} where necessary.  In the case $V\neq0$, by Lemma 3.4.2 of \cite{A-B-C-F-G-M-R-S}, we may write
\begin{align*}
\Ent_{\hat\mu}^\Phi(|f|) &= \inf_{t\in[0,\infty)} \hat\mu\left(\Phi(|f|) - \Phi'(t)(|f|-t) - \Phi(t)\right)\\
&\leq \frac{e^{osc(V)}Z_0}{\hat Z}\inf_{t\in[0,\infty)} \int\left(\Phi(|f|) - \Phi'(t)(|f|-t) - \Phi(t)\right)\frac{e^{-W}}{Z_0}d\mu
\end{align*}
where $Z_0=\int e^{-W}d\mu$.  Applying the above case when $V=0$ to the measure $\frac{e^{-W}}{Z_0}d\mu$ yields
\begin{align*}
\Ent_{\hat\mu}^\Phi(|f|) &\leq \frac{e^{osc(V)}Z_0}{\hat Z}c'\int|\nabla f| \frac{e^{-W}}{Z_0}d\mu \\
&\leq c'e^{2osc(V)}\hat{\mu}|\nabla f|,
\end{align*}
for some constant $c'$, so that the result holds.
\end{proof}

In Theorem \ref{tighten} we assume that the Cheeger type inequality \eqref{cheeger} holds, together with inequalities \eqref{class sob} and \eqref{U-bound}.  However, we note below that under some conditions it is possible to deduce the Cheeger type inequality directly from a weaker version of the $U$-bound \eqref{U-bound}, using the method in \cite{H-Z}.

\begin{theorem}\label{thm1.6}
\label{cheeger from U-bound}
Let $d\mu = \frac{e^{-U}}{Z}d\lambda$ be probability measure on $\R^N$, and suppose that the following inequality is satisfied
\begin{equation}
\label{weak U-bound}
\mu\left(f |U|^\beta\right)\leq A\mu|\nabla f| + B\mu|f|,
\end{equation}
for some $\beta>0$.  Suppose also that
\begin{itemize}
\item[{\rm (a)}] for any $L\geq0$ there exists $r=r(L)\in(0,\infty)$ such that 
\begin{equation}\label{compactness assumption}
\left\{|U|^\beta\leq L\right\} \subset B(r) 
\end{equation}
for some ball $B(r)$ of radius $r$;
\item[{\rm (b)}] for $r=r(L)$ there exists $m_r\in(0,\infty)$ such that the following Poincar\'e inequality in the ball $B(r)$ is satisfied
\begin{equation}
\label{Poincare in ball}
\int_{B(r)} \left|f - \frac{1}{\lambda(B(r))}\int_{B(r)} f d\lambda\right|\, d\lambda\, \leq\, \frac1{m_r}\int_{B(r)}|\nabla f|d\lambda.
\end{equation}
\end{itemize}
Then there exists a constant $c_0$ such that
\[
\mu|f- \mu f| \leq c_0\mu|\nabla f|.
\]
\end{theorem}

\begin{proof}
We have that
\[
\mu|f - \mu f| \leq 2\mu|f -m|
\]
for all $m\in\R$.  Now for $L\geq0$ we have
\begin{equation}
\label{split2}
\mu|f - m| \leq \mu\left(|f-m|\Ind_{\{|U|^\beta \leq L\}}\right) + \mu\left(|f-m|\Ind_{\{|U|^\beta \geq L\}}\right).
\end{equation}
We have that $\{|U|^\beta \leq L\}\subset B(r)$ for some $r=r(L)\in(0, \infty)$, so that putting $m = \frac{1}{\lambda(B(r))|}\int_{B(r)}fd\lambda$, and noting that on the set $\{|U|^\beta \leq R\}$ there exists a constant $A_r$ such that 
\[
\frac{1}{A_r} \leq \frac{d\mu}{d\lambda} \leq A_r,
\]
we can bound the first term using assumption (a).  Indeed,
\begin{align}
\label{ball}
\mu\left(|f-m|\Ind_{\{|U|^\beta \leq L\}}\right) &\leq A_r\int_{B(r)}\left|f - \frac{1}{\lambda(B(r))}\int_{B(r)} f d\lambda\right|d\lambda\nonumber\\
&\leq  \frac{A_r}{m_r} \int_{B(r)}|\nabla f|d\lambda \leq  \frac{A_r^2}{m_r}\mu|\nabla f|
 \end{align}
using \eqref{Poincare in ball}. On the other hand,  using \eqref{weak U-bound}, we have
 \begin{align}
 \label{outside}
 \mu\left(|f-m|\Ind_{\{|U|^\beta \geq L\}}\right) & \leq \frac{1}{L}\mu\left(\left|f - m\right||U|^\beta\right)\nonumber\\
  & \leq \frac{A}L\mu|\nabla f| + \frac{B}{L}\mu|f - m| .
 \end{align}
Using estimates \eqref{ball} and \eqref{outside} in \eqref{split2}, and taking $L$ large enough ends the proof.
\end{proof}

We can now combine all the results of this section into the following Theorem.

\begin{theorem}\label{thm1.7}
\label{main}
Let $U$, $\lambda$ and $\mu$ be as in Theorem \ref{defective beta LS}.
Suppose also that conditions {\rm (a)} and {\rm (b)} of Theorem \ref{cheeger from U-bound} are satisfied. Then  there exists $c\in(0,\infty)$ such that \eqref{beta-LS} holds, i.e.
\[
\mathbf{Ent}_\mu^\Phi(|f|) \leq c\mu|\nabla f|,
\]
where $\Phi(x) = x\left(\log(1+x)\right)^\beta$.
\end{theorem}

To conclude this section, we finally note that the $L_1\Phi$-entropy inequality \eqref{beta-LS} can be tensorised in the following sense.

\begin{lemma}[Tensorisation]\label{lem1.8}
\label{tensorisation}
Let $I$ be a finite index set, and $\nu_{i}, i\in I$ be probability measures.  Set $\nu_I := \otimes_{i \in I}\nu_i$.  Suppose that for each $i\in I$, $\nu_i$ satisfies the $L_1\Phi$-entropy inequality \eqref{beta-LS} with a constant $c(i)\in(0,\infty)$.  Then so does $\nu_I$ with constant $\max_{ i\in I}\{c(i)\}$.
\end{lemma}

\begin{proof}
The proof follows by induction. The key observation is as follows: for $J\subset I$ and $k\notin J$, one has
 \begin{align}
\nu_k\otimes\nu_J\Phi(f)-\Phi(\nu_k\otimes\nu_Jf) &=
\nu_k\left(\nu_J\Phi(f) -\Phi(\nu_Jf)\right)+\left(\nu_k\Phi(\nu_Jf)-\Phi(\nu_k(\nu_Jf))\right)
\nonumber \\
&\leq \nu_k\left(\sum_{j\in J}c_J\nu_J |\nabla_jf|\right) + c_k \nu_k|\nabla_k \nu_Jf| \nonumber \\
&\leq \max(c_J,c_k)\sum_{j\in J\cup k}\nu_k\otimes\nu_J |\nabla_jf| \nonumber .
\end{align}
\end{proof} 


\section{Isoperimetric inequalities} \label{Sec3.Isoperimetric Inequalities}
In this section our aim is to derive isoperimetric
information for the measure \(\mu\) starting from $L_1\Phi$-entropy inequalities. We assume that \(\mu\) is non-atomic and that the distance \(d\) on \(\R^N\)  is related to the modulus of the gradient of a function \(f:\R^N
\rightarrow \R\) by  \begin{equation}\label{metricgrad}\vert
\nabla f \vert(x) = \limsup_{d(x,y)\downarrow 0}\frac{\vert f(x)-f(y)\vert}{d(x,y)}.
\end{equation} 
As usual, we define the surface measure of a Borel set \(A\subset \R^N\) by \begin{equation*}
\mu^+(A)=\liminf_{\varepsilon \downarrow0} \frac{\mu(A^\varepsilon\setminus
A)}{\varepsilon}
\end{equation*} 
where \(A^\varepsilon= \{x \in \R^n : d(x,A)<\varepsilon\}\) is the (open) \(\varepsilon\)-neighbourhood of \(A\) (with respect to \(d\)). We are concerned with a problem of estimating the isoperimetric profile of the measure \(\mu\), that is a function \(\mathcal{I}_{\mu}:[0,1]\rightarrow \R^+ \) defined by \[\mathcal{I}_\mu(t)=\inf\{\mu^+(A):
A  \text{ Borel such that } \mu(A)=t\}\] 
(with \(\mathcal{I}_\mu(0)=\mathcal{I}_\mu(1)=0) \). By definition it is the largest function such that the following isoperimetric inequality holds 
\begin{equation}
\label{best}
\mathcal{I}_\mu(\mu(A)) \le\mu^+(A).
\end{equation}

For \(q >1\) and \(p\) such that \(\frac{1}{q}+\frac{1}{p}=1\),  we define  functions \(\mathcal{U}_q=f_p \circ F_p^{-1}\) where \(f_p\)
is the density of the measure \(d\nu_p(x)=\frac{e^{-\vert x\vert^p}}{Z_p}dx\) on $\R$ and
\(F_p'=f_p\) (here, \(\vert x \vert\) denotes the Euclidean norm of \(x \in
\R)\). This is motivated by the fact that \(\mathcal{U}_q\) is the  isoperimetric
function of \(\nu_p\) in dimension $1$. It is known (see  \cite{B-Z}) that   \(\mathcal{U}_{q}(t)\)
is symmetric and behaves like \(G(t)=t \left( \log \left( \frac{1}{t} \right) \right)^\frac{1}{q}\)
near the origin so that for some constant \(L_q>0\), we have

\begin{align}\label{qequivalence}
\frac{1}{L_{q}}G(\min(t,1-t))\le \mathcal{U}_q(t) \le L_{q}G(\min(t,1-t))
\end{align}
for all $t\in[0,1]$.

\begin{theorem}\label{thm3.1}\label{isoq}
Assume that the $L_1\Phi$-entropy inequality 
\[
\mathbf{Ent}_\mu^\Phi(|f|) \leq c\ \mu|\nabla f|
\]
holds for some constant $c\in(0, \infty)$ and all locally Lipschitz
functions \(f\), where $\Phi(x) = x\left(\log(1+x)\right)^\beta$ and $\beta\in(0,1]$. 
Then \(\mathcal{I}_\mu \ge \frac{1}{\tilde{c}}\ \mathcal{U}_q\) with some constant \(\tilde{c}>0\), 
\(q=\frac{1}{\beta}\) and the measure \(\mu\) satisfies
an isoperimetric inequality of the form
 \begin{equation}\label{isoqineq}
\mathcal{U}_q(t) \le \tilde{c}\ \mu^+(A)
\end{equation}
for all a Borel sets \(A\) of measure \(t=\mu(A) \).
\end{theorem}

\proof When applied to a nonnegative function \(f\) such that \(\mu f=1\), the $L_1\Phi$-Entropy inequality becomes 
\begin{equation*}
\mu\left( f\left((\log(1+f))^\beta-(\log 2)^\beta\right)\right)\le c \mu \vert
\nabla f \vert, 
\end{equation*}
which implies that for all non-negative \(f\) (not identically 0) we have  
\begin{align}
\label{homogeneous phi entropy}
\mu\left( f\left(\left(\log\left(1+\frac{f}{\mu f}\right)\right)^\beta-(\log 2)^\beta\right)\right)\le c \mu \vert
\nabla f \vert. 
\end{align}
Let $A$ be a Borel set with measure $t=\mu(A)$.  To start with, suppose that $t\in\left[0, \frac{1}{2}\right]$.  We can approximate the indicator function of $A$ by a sequence of Lipschitz
functions \((f_{n})_{n\in\mathbb{N}} \) satisfying 
\[
\limsup_{n \rightarrow \infty} \mu \vert
\nabla f_n \vert \le \mu^+(A)
\]
(see \cite{bobkovhoudre}, Lemma 3.5). Taking $f_n$ in \eqref{homogeneous phi entropy} and passing to the limit as $n\to\infty$ yields
 \begin{equation}
t\left( \left( \log \left(1+ \frac{1}{t} \right) \right)^\beta  - \left( \log 2 \right)^\beta \right) \le c\mu^{+}(A).
\end{equation} 
We now  observe that for \(t \in \left[0,\frac{1}{2}\right] \) we have 
\begin{equation}
\label{etaestimate}
\eta \left( \log \left( \frac{1}{t} \right) \right)^\beta \le \left( \log \left( 1+\frac{1}{t} \right) \right)^\beta - ( \log 2 )^\beta
\end{equation}
with \(\eta= \left( \frac{\log 3}{ \log 2} \right)^\beta-1>0\). This implies
\begin{equation}\label{preiso}
t \left( \log \left( \frac{1}{t} \right) \right)^\beta \le \frac{c}{\eta}\ \mu^+(A),
\end{equation}
for all $t\in\left[0, \frac{1}{2}\right]$. 
Thus, by the equivalence relation \eqref{qequivalence}, we have that
\begin{equation}
\label{0tohalf}
\mathcal{U}_q(t) \le \tilde{c}\ \mu^+(A)
\end{equation}
for all $t\in[0,\frac{1}{2}]$, with \(\tilde{c}=\frac{c}{\eta}L_q\).

Now suppose that $t= \mu(A)\in\left(\frac{1}{2}, 1\right]$.  For functions $f\in [0, 1]$, we can apply \eqref{homogeneous phi entropy} to $1-f$, which yields
\begin{equation*}
\mu\left( (1-f)\left(\left(\log\left(1+\frac{1-f}{1-\mu f}\right)\right)^\beta-(\log 2)^\beta\right)\right)\le c\ \mu \vert
\nabla f \vert. 
\end{equation*}
If we now take $f_n$ in this inequality (where $(f_n)_{n\in\mathbb{N}}$ is again the Lipschitz approximation of the characteristic function of $A$) and pass to the limit as $n\to\infty$, we see that
 \begin{equation*}
(1-t)\left( \left( \log \left(1+ \frac{1}{1-t} \right) \right)^\beta  - \left( \log 2 \right)^\beta \right) \le c\ \mu^{+}(A).
\end{equation*}
Writing \(s=1-t \in \left[0,\frac{1}{2} \right)\) and using \eqref{etaestimate} now gives 
\begin{equation}
s \left( \log \left( \frac{1}{s} \right) \right)^\beta \le \frac{c}{\eta}\ \mu^+(A).
\end{equation} 
Thus by \eqref{qequivalence} again, we have $\mathcal{U}_q(1-t) = \mathcal{U}_q(s) \le \tilde{c} \mu^+(A)$ for all $t\in\left(\frac{1}{2}, 1\right]$ with $\tilde{c} = \frac{c}{\eta}L_q$.  By symmetry of \(\U_q\) therefore $\mathcal{U}_q(t) \le \tilde{c} \mu^+(A)$ for $t\in\left(\frac{1}{2}, 1\right]$, which combined with \eqref{0tohalf} yields the result.
\endproof

An important corollary of this result is the following:

\begin{corollary} \label{cor3.2}
\label{phi-ent to cheeger}
Assume that the $L_1\Phi$-entropy inequality 
\[
\mathbf{Ent}_\mu^\Phi(|f|) \leq c\ \mu|\nabla f|
\]
holds for some constant $c\in(0, \infty)$ and all locally Lipschitz
functions \(f\), where $\Phi(x) = x\left(\log(1+x)\right)^\beta$ and $\beta\in(0,1]$. Then there exists a constant $c_0$ such that
\begin{equation}
\label{cheeger01}
\mu|f-\mu f| \leq c_0 \mu|\nabla f|.
\end{equation}
\end{corollary}

\begin{proof}
We note that
if \(\beta=1/q,\)\[\mathcal{U}_q(t)
\ge \frac{1}{L_q}\min(t,1-t)\log\left(
\frac{1}{\min(t,1-t)}\right)^{1/q} \ge\frac{(\log2)^{1/q}}{L_q}\min(t,1-t).
\]
Thus by Theorem \ref{isoq}, we have that 
\[
\min(t,1-t) \leq \tilde{c}\frac{L_q}{(\log2)^\beta}\mu^+(A),
\]
for $t=\mu(A)$, which is Cheeger's isoperimetric inequality on sets. This is equivalent (up to a constant) to its functional form
\[ 
\mu \vert f-\mu f\vert \leq c_0\, \mu \vert \nabla f \vert
\]
(see for example \cite{bobkovhoudreams}).
\end{proof}

Following an argument of \cite{Led1}  we can pass from the isoperimetric statement above to inequality \eqref{defective1}. 
We note that in our general setting, the following coarea inequality
is available, (for a proof see e.g. \cite{bobkovhoudre}, Lemma 3.2),

\begin{equation}\label{coarea}
\mu \vert \nabla f \vert  \ge \int_{\R} \mu^+(\{f> s\})ds
\end{equation}
for locally Lipschitz functions \(f\).

\begin{proposition}\label{pro3.3}
\label{tightledouxprop} If the measure \(\mu\) satisfies an isoperimetric inequality
of the form \eqref{isoqineq}, then there exist  constants $K, K'>0$ such that
\begin{equation}\label{tightledoux} \mu \left( f ( \log_{+} f )^\beta \right) \le K\mu
\vert \nabla f \vert + K'
\end{equation}
for all positive locally Lipschitz functions \(f\) such that \(\mu (f)=1\), where \(\beta=\frac{1}{q}\). 
\end{proposition}

\proof
Let \(f\) be non-negative, with $\mu(f)=1$. The coarea inequality \eqref{coarea} together with our assumption imply 
\begin{align*}
\mu \vert \nabla f \vert &\ge \int_{\R} \mu^+(\{f> s\})ds \ge\frac{1}{\tilde{c}} \int_{\R} \mathcal{U}_q(\mu(\{f>s\}))ds
\end{align*}
Let us note that \begin{align*}
\int_0^1\min(t,1-t) \left( \log \frac{1}{\min(t,1-t)} \right)^\beta dt&=2\int_0^{1/2}t
\left( \log \frac{1}{t} \right)^\beta dt \\ & \ge 2\int _0^1 t 
\left( \log \frac{1}{t} \right)^\beta dt - M
\end{align*}
where \(M= \sup _{t \in \left(\frac{1}{2},1\right)}t
\left( \log \frac{1}{t} \right)^\beta \). By \eqref{qequivalence}, we conclude that \[ \mu \vert \nabla f \vert  \ge   K \int_{\R} \mu(\{f > s\}) \log \left(
\frac{1}{\mu(\{f > s\})}\right)^\beta ds - \frac{MK}{2}
\]
with $K\equiv\frac{2}{\tilde{c}L_{q}}$.
By Markov'’s inequality, $ \mu(\{f > s\})\leq \frac1s$.
Therefore, when $s\geq 1$ we have
\[\log\frac{1}{\mu(\{f > s\})}\geq \log s\] 
and we always have $\log \frac{1}{\mu(\{f > s\})} \geq 0$.
Therefore, \( \log\frac{1}{\mu(\{f > s\})}\ge\log_+s,  \) which implies 
\[ \mu \vert \nabla f \vert  \ge K   \int_{\R} 
\left( \log_+   s \right)^\beta \mu(\{f > s\})  ds- \frac{MK}{2} \ge K  \mu \left( f \left( \log_+   f \right)^\beta\right) - K'
\]
with some constant $K'\in(0,\infty).$
\endproof

\begin{rem}
With the above results, we have thus shown the equivalence of the $L_1 \Phi$-entropy inequality with the isoperimetric inequality \eqref{isoqineq} and with inequality \eqref{tightledoux} together with the Cheeger inequality \eqref{cheeger01}; see Theorem \ref{summary} below.
\end{rem}

\begin{rem}
When \(\frac{1}{\beta}=q=2\), the function \(\mathcal{U}_2\) represents the
Gaussian isoperimetric function. In this case, the isoperimetric inequality
\eqref{isoqineq} is known to be equivalent to the  following inequalities introduced by Bobkov in \cite{bobkovifi} and \cite{bobkovcube}:\begin{align}\label{1}&
\;\mathcal{U}_{2}(\mu(f)) \le \mu\left( \mathcal{U}_2 (f)+ \tilde{c}\vert\nabla f \vert \right) \\ \label{2}&\;\mathcal{U}_{2}(\mu(f)) \le \mu \left( \sqrt{\mathcal{U}_2 (f)^2+ \tilde{c}^2\vert \nabla f \vert^2} \right)
\end{align} 
for all locally Lipschitz \(f:\R \rightarrow [0,1]\). The equivalence of
these inequalities in this case follows by a transportation argument which uses the fact that the standard Gaussian
measure \(\gamma\) on \(\R\) satisfies \eqref{1} and \eqref{2} with \(\tilde{c}=1\)
(see \cite{barthemaurey}, Proposition 5).  \end{rem}

\begin{rem} Suppose that the measure \(\mu\) satisfies an \(L_1\Phi\)-entropy
inequality on a metric space \((\mathcal{M},d)\). Suppose that on the product space \((\mathcal{M}^{n},d_{n},\mu^{\otimes
n})\)
we have \(
\vert \nabla f \vert= \sum_{i=1}^n \vert \nabla_i f \vert
\), where \( \nabla_i\) denotes differentiation with respect to the \(i^{th}\)
coordinate and where the moduli of the gradients are defined via \eqref{metricgrad} with the supremum distance. The tensorisation property of the \(L_1\Phi\)-entropy (Lemma \ref{tensorisation})
then allows us to obtain
isoperimetric information on the product space (where the surface measure
is now defined with respect to supremum distance). This problem was considered
in \cite{Barthesupremum}.  \end{rem}

\section{Consequences of $L_1\Phi$-entropy inequalities}
\label{Consequences of L^1 Phi-entropy inequalities}

In this section we look at some consequences of the $L_1\Phi$-entropy inequality
\begin{equation}
 \label{beta LS 2}
 \Ent^\Phi_\mu(|f|) \leq c \mu|\nabla f|,
\end{equation}
with $\Phi(x) = x(\log(1+x))^\beta$, $\beta\in(0,1]$, for a general probability measure $\mu$.  The first result shows that this inequality implies a $q$-logarithmic Sobolev inequality, as studied in \cite{B-Z} and \cite{H-Z}.

\begin{theorem}\label{thm2.1}
\label{Phi-entropy to LSq}
Let $\mu$ be an arbitrary probability measure which satisfies the $L_1\Phi$-entropy inequality \eqref{beta LS 2} for some $\beta\in[\frac{1}{2},1]$ and set $q = \frac{1}{\beta}\in[1,2]$.  Then there exists a constant $c_q$ such that the following $(LS_q)$ inequality holds
\begin{equation}
\label{defective LSq}
\mu\left(|f|^q\log\frac{|f|^q}{\mu|f|^q}\right) \leq c_q \mu|\nabla f|^q.
\end{equation}

\end{theorem}

\begin{proof}
Without loss of generality we assume that $f\geq0$.  Applying  $L_1\Phi$-entropy inequality \eqref{beta LS 2} to the function $f/\mu f$, we obtain the following homogeneous version
\begin{equation}
\label{hom Phi-ent}
\mu\left(f\left[\log\left(1+\frac{f}{\mu f}\right)\right]^\beta\right) \leq c\mu|\nabla f| + (\log 2)^\beta\mu(f).
\end{equation}
We apply this inequality to the function $g = f\left(1 + \log(1+ f)\right)^{1-\beta}\geq f\geq0$, where $f$ is such that $\mu(f) =1$.  Note that $\mu(g)\geq1$.  Then we have
\begin{align*}
\mu\left(g\left[\log\left(1+\frac{g}{\mu g}\right)\right]^\beta\right) 
&=   \mu\left( f\left(1 + \log(1+ f)\right)^{1-\beta}
\left[
\log\left(1+ \frac{g}{\mu g}\right)\right]^\beta \right) \\
& \geq   \mu\left( f\left(1 + \log(1+ f)\right)^{1-\beta}\left[\log\left(1+ \frac{f}{\mu g}\right)\right]^\beta\right) \\
& \geq   \mu\left( f\left(1 + \log\left(1+ \frac{f}{\mu g}\right)\right)^{1-\beta}\left[\log\left(1+ \frac{f}{\mu g}\right)\right]^\beta\right) \\
&\geq \mu\left( f\log\left(1+ \frac{f}{\mu g}\right)\right) 
 = \mu\left( f\log (\mu g + f)\right) - \log\mu(g)\\
&\geq \mu\left( f\log (1+f)\right) - \mu(g).
\end{align*}
Thus for all $f\geq0$ with $\mu(f)=1$,
\begin{align}
\label{normalised}
 \mu\left( f\log (1+f)\right) &\leq c\mu\left|\nabla\left(f\left(1 + \log(1+ f)\right)^{1-\beta}\right)\right| + \left((\log2)^\beta + 1\right)\mu(g)\nonumber\\
 &\leq c\mu\left(\left(1 + \log(1+ f)\right)^{1-\beta}|\nabla f|\right) \nonumber\\
 &\quad + c(1-\beta)\mu\left(\frac{f}{\left(1+ \log(1+f)\right)^\beta}\frac{1}{1+f}|\nabla f|\right) + \left((\log2)^\beta + 1\right)\mu(g)\nonumber\\
 &\leq c\mu\left(\left(1 + \log(1+ f)\right)^{1-\beta}|\nabla f|\right) + c(1-\beta)\mu|\nabla f|\nonumber\\
 &\qquad + \left((\log2)^\beta + 1\right)\mu(g).
 \end{align}

Since we have assumed $\beta\geq\frac{1}{2}$, we have $1-\beta\leq \beta$ and hence
 \begin{align*}
 \mu(g) = \mu\left(f\left(1+\log(1+f)\right)^{1-\beta}\right) &\leq 1 + \mu(f[\log(1+ f)]^{1-\beta})\\
 &\leq  \mu(f[\log(1+ f)]^{\beta})+2\\
 &\leq c\mu|\nabla f| + (\log 2)^\beta + 2
 \end{align*}
by another application of the $L_1\Phi$-entropy inequality \eqref{hom Phi-ent} in the last step.
Using this in \eqref{normalised}, we see that for general $f\geq 0$,
\begin{align}
\label{general f}
\mu\left( f\log\left(1+\frac{f}{\mu f}\right)\right) &\leq c\mu\left(\left(1 + \log\left(1+ \frac{f}{\mu f}\right)\right)^{1-\beta}|\nabla f|\right)\nonumber \\
 &\quad + c(2-\beta+(\log2)^\beta)\mu|\nabla f| + \left((\log 2)^\beta+2\right)^2\mu(f).
 \end{align}
Replacing $f$ by $f^q$ with $q=\frac{1}{\beta}$ in the above yields
\begin{align*}
\mu\left( f^q\log\left(1+\frac{f^q}{\mu f^q}\right)\right) &\leq qc\mu\left(\left(1 + \log\left(1+ \frac{f^q}{\mu f^q}\right)\right)^{1-\beta}f^{q-1}|\nabla f|\right)\nonumber \\
&\quad + cq(2-\beta + (\log2)^\beta)\mu\left(f^{q-1}|\nabla f|\right) 
+ \left((\log 2)^\beta+2\right)^2\mu(f^q)\\
&\leq \frac{qc\varepsilon^{p-1}}{p}\mu\left(f^q\left(1 + \log\left(1+ \frac{f^q}{\mu f^q}\right)\right)\right)\\
&\quad + \left(\frac{c}{\varepsilon} + c(2-\beta+ (\log2)^\beta )\right)\mu|\nabla f|^q\\
&\quad + \left(\frac{cq}{p}(2-\beta+ (\log2)^\beta) + \left((\log 2)^\beta + 2\right)^2\right)\mu(f^q) 
\end{align*}
where  $\varepsilon>0$ and we have applied Young's inequality  with indices $\frac{1}{p}+\frac{1}{q}=1$. 
Choosing $qc\varepsilon^{p-1}/p<1$, we can simplify this bound as follows
\[
\mu\left( f^q\log\left(1+ \frac{f^q}{\mu f^q}\right)\right) \leq C' \mu|\nabla f|^q + D'\mu(f^q)
 \]
 where
\[
C' = \frac{\frac{c}{\varepsilon} + c(2-\beta+ (\log2)^\beta )}{1-\frac{qc\varepsilon^{p-1}}{p}}, \qquad D' = \frac{\frac{cq}{p}(2-\beta+ (\log2)^\beta) + \left((\log 2)^\beta +2\right)^2 }{1-\frac{qc\varepsilon^{p-1}}{p}}.
 \]
From this one obtains the defective $LS_q$,   which for all $f\geq0$ such that $\mu(f^q)=1$ can be equivalently represented as
 \begin{equation}
 \label{normalisation2}
 \mu\left(f^q\log f^q\right) \leq C'\mu|\nabla f|^q + D'.
 \end{equation} 
Let us now recall that by Corollary \ref{phi-ent to cheeger}, our assumption
implies that there exists a constant $c_0$ such that
\begin{displaymath}
\mu \vert f-\mu f \vert  \le c_0 \mu \vert \nabla f \vert.
\end{displaymath}
From this inequality we can use the arguments of \cite{B-Z} (Chapter 2)  to deduce that there exists a constant $c_q$ such that
\[
\mu \vert f-\mu f \vert^q  \le c_q \mu \vert \nabla f \vert^q.
\]
Finally, by Rothaus type arguments (see \cite{B-Z} Chapter 3), we can then remove the defective term in
 \eqref{normalisation2} to arrive at the result.\end{proof}

Theorem \ref{Phi-entropy to LSq} has a number of corollaries, which follow from known results about the $q$-logarithmic Sobolev inequality $(LS_q)$ contained in \cite{B-Z} and \cite{H-Z}.  We mention here the following one, which is important for our purposes.

\begin{corollary}\label{cor2.4}
\label{exponential bounds}
Let $\mu$ be an arbitrary probability measure which satisfies the $L_1\Phi$-entropy inequality \eqref{beta LS 2} with $\beta\in[\frac{1}{2},1]$.  Suppose $f$ is a locally Lipschitz function such that 
\begin{equation}
\label{gradient condition}
|\nabla f|^q \leq af + b
\end{equation}
with $q = \frac{1}{\beta}$, for some constants $a, b\in[0,\infty)$.  Then for all $t>0$ sufficiently small
\[
\mu\left(e^{tf}\right) <\infty.
\]
\end{corollary}

\begin{proof}
Follows from Theorem 4.5 of \cite{H-Z}.
\end{proof}

In Section \ref{L_1 entropy inequalities from U-bounds} we proved that, under some conditions, if $d\mu = \frac{e^{-U}}{Z} d\lambda$ is a probability measure which satisfies a Cheeger type inequality of the form \eqref{cheeger}, and a $U$-bound of the form
\begin{equation}
\label{beta U-bound}
 \mu\left(|f|\left[|U|^\beta + |\nabla U|\right]\right) \leq A\mu|\nabla f| + B\mu|f|,
 \end{equation}
then the $L_1\Phi$-entropy inequality \eqref{beta LS 2}  holds.
 
We now aim to show the converse i.e. that under some weak conditions, the $L_1\Phi$-entropy inequality \eqref{beta LS 2} implies a bound of the form \eqref{beta U-bound}.  We first prove the following useful lemma.
 
\begin{lemma}\label{lem2.5}
 \label{relative Phi-entropy lem}
Let $\mu$ be a probability measure. Then 
\begin{equation}
\label{relative Phi-entropy}
\mu(fh) \leq s^{-1} \mathbf{Ent}_\mu^\Phi(f) + s^{-1}\Theta(sh)
\end{equation}
for all $s>0$ and suitable functions $f, h\geq0$ such that $\mu(f)=1$, where $\Phi(x) = x\left(\log(1+x)\right)^\beta, \beta\equiv \frac1q\in(0,1]$ and
\[
\Theta(h) \equiv  \left(\theta+(\log 2)^\beta + \left(\log \mu e^{ h^q}\right)^\beta \right)
\]
with $\theta = \sup_{x\geq0} \beta x(\log(1+x))^{\beta-1}/(1+x)$.

Moreover, suppose that $\mu$ satisfies the $L_1\Phi$-entropy inequality \eqref{beta LS 2} for some $\beta\in[\frac{1}{2},1]$ with constant $c$, and that $g\geq0$ is a locally Lipschitz function such that
\begin{equation}
\label{g gradient condition}
|\nabla g|^q \leq ag + b
\end{equation}
for some constants $a, b\in(0,\infty)$.  Then $\Theta(s^\beta g^\beta)<\infty$ for sufficiently small $s>0$, and
\begin{equation}
\label{relative Phi-entropy 2}
\mu(fg^\beta) \leq \frac{c}{s^\beta}\mu|\nabla f| + \frac{c}{s^\beta}\Theta(s^\beta g^\beta)\mu(f),
\end{equation}
for all functions $f\geq0$ for which the right hand side is well defined.
\end{lemma}

\begin{proof}  \label{prf_lem2.5}
We remark first that for  functions $f,h\geq 0$, $\mu f = 1$ , with $s\in(0,\infty)$ and 
$\beta\equiv \frac1q\in(0,1)$, we have
\begin{align*}
\mu(fh) = s^{-1}\mu \Big(f\left(\log e^{s^qh^q}\right)\Big)^\beta & \leq
s^{-1}\mu \left[f\left(\log\left( 1+\frac{e^{s^qh^q}}{\mu e^{s^qh^q}}\right)\right)^\beta \chi(e^{s^qh^q}\geq \mu e^{s^qh^q} )\right]
\\ 
&+ s^{-1}\left(\log \mu e^{s^qh^q}\right)^\beta  \mu (f).
\end{align*}
By the generalised relative entropy inequality of \cite{F-R-Z}, we have
\begin{align*}
\mu \left[f\left(\log\left( 1+\frac{e^{s^qh^q}}{\mu e^{s^qh^q}}\right)\right)^\beta \right]
&\leq  \mu f\left(\log\left( 1+\frac{f}{\mu f}\right)\right)^\beta +\theta \mu f \\
&\leq \mathbf{Ent}_\mu^\Phi (f) + (\theta+(\log 2)^\beta) \mu f,
\end{align*}
since $\mu f=1$.  We therefore get the following bound
\begin{equation}\label{eq2.5.1}
\mu(fh)    \leq
s^{-1}\mathbf{Ent}_\mu^\Phi (f)  
 + s^{-1}\left(\theta+(\log 2)^\beta + \left(\log \mu e^{s^qh^q}\right)^\beta \right).
\end{equation}
This ends the proof of the first part of the lemma. \\
Replacing $h$ by $g^\beta\equiv g^{\frac1q}$  and $s$ by $s^\beta$ in \eqref{eq2.5.1}, we see that the second part is a consequence of Corollary \ref{exponential bounds}. 
\end{proof}

\begin{theorem}\label{thm2.6}
\label{converse}
Let $d\mu = \frac{e^{-U}}{Z}d\lambda$ be a probability measure on $\R^N$, with $U$ a locally Lipschitz function bounded from below. Suppose $\mu$ satisfies the $L_1\Phi$-entropy inequality \eqref{beta LS 2} for some $\beta \in[\frac{1}{2},1]$. \\ 
Suppose also that
\begin{equation}
\label{U gradient condition}
|\nabla U| \leq a|U|^\beta + b
\end{equation}
for some constants $a, b\in(0, \infty)$.
Then there exist constants $A,B\in[0, \infty)$ such that
\begin{equation}
\label{converse U bound}
\mu\left(|f|\left(|U|^\beta + |\nabla U|\right)\right) \leq A\mu|\nabla f| + B\mu|f|,
\end{equation}
for all $f$ for which the right-hand side is well defined.
\end{theorem}

\begin{proof}
Let $f\geq0$.  We may also suppose that $U\geq0$ (otherwise we can shift it by a constant).  Note that from \eqref{U gradient condition}, it follows that  
\[
|\nabla U|^q \leq \tilde{a}U + \tilde{b}
\]
with $q=\frac{1}{\beta}$.  Hence we may apply Lemma \ref{relative Phi-entropy lem}, to see that
\[
\mu(fU^\beta) \leq \frac{c}{s^\beta}\mu|\nabla f| + \frac{c}{s^\beta}\Theta(s^\beta U^\beta)\mu(f)
\]
with $\Theta(s^\beta U^\beta)<\infty$ for sufficiently small $s$.
\end{proof}


The following Theorem summarises the results of the paper so far.

\begin{theorem}\label{thm3.4}\label{summary}
Let $\mu$ be a non-atomic  probability measure on $(\R^N, d)$,   \(\vert \nabla f \vert\) be given by \eqref{metricgrad} and  $q\geq1$.  Then the following statements are equivalent
\begin{itemize}
\item[{\rm (i)}]
\[\mathbf{Ent}_\mu^\Phi(|f|) \leq c\mu|\nabla f|,
\]
where $\Phi(x) = x\left(\log(1+x)\right)^\frac{1}{q}$, for some constant $c\in(0,\infty)$  and all locally Lipschitz \(f\);
\\

\item[{\rm (ii)}]
\[
\mu \left(   f   \left( \log_+ \frac{ f  }{\mu  f  } \right)^{1/q} \right) \le 
K\, \mu \vert \nabla f \vert+K'\mu f,
\]
for some $K>0$  and
\[\mu \vert f-\mu f\vert \leq c_0\, \mu \vert \nabla f \vert
\]
with some $c_0\in(0,\infty)$ and all locally Lipschitz \(f\geq 0\);
\\

\item[{\rm (iii)}]
\[
\mathcal{U}_q(t) \le \tilde{c} \mu^+(A),
\]
for some \(\tilde{c}>0\) and all Borel sets \(A\) of measure \(t=\mu(A) \).
\\

\end{itemize}
Moreover, for $q\in(1,2]$ statements {\rm (i) - (iii)} imply
\begin{itemize}
\item[{\rm (iv)}]
\[
\mu\left(|f|^q\log\frac{|f|^q}{\mu|f|^q}\right) \leq C'\mu|\nabla f|^q \eqno{(LS_q)}
\]
for some $C'\in(0,\infty)$ and all locally Lipschitz functions $f$
\end{itemize}
and
\begin{itemize}
\item[{\rm (v)}]
\[ \mathcal{U}_2(\mu f) \leq \mu \sqrt{\mathcal{U}_2^2(f)+ C''|\nabla f|^2}
 \eqno{(IFI_2)}
\]
for some $C''\in(0,\infty)$ and all locally Lipschitz functions $0\leq f\leq 1$.
\end{itemize}

Finally, suppose that the probability measure $\mu$ is given by $\mu(dx) = \frac{e^{-U}}{Z}d\lambda$ for some locally Lipschitz function $U$ on $\R^N$ which is bounded from below.  Suppose that the measure $d\lambda$ satisfies the classical Sobolev inequality \eqref{class sob} together with the Poincar\'e inequality in balls \eqref{Poincare in ball}, and that $\forall L\geq0$ there exists $r=r(L)$ such that $\{U\leq L\}\subset B(r)$.  In this situation the following $U$-bound

\begin{equation}
\label{U bound again}
\mu\left(|f|\left(|U|^\beta + |\nabla U|\right)\right) \leq A\mu|\nabla f| + B\mu|f|
\end{equation}
for constants $A, B\in[0, \infty), \beta\in(0,1]$, implies that statements {\rm (i)-(iii)} hold with $q = \frac{1}{\beta}$.  If in addition we have that \eqref{U gradient condition} holds i.e.  there exist constants $a,b$ such that
\[
|\nabla U| \leq aU^\beta + b
\]
then \eqref{U bound again} is actually equivalent to the statements {\rm(i) - (iii)}.
\end{theorem}

\begin{proof}
(ii)$ \Rightarrow $ (i) was shown in Section \ref{L_1 entropy inequalities from U-bounds}.  (i) $\Rightarrow$ (iii) is proved in Theorem \ref{isoq}. Finally, Proposition \ref{tightledouxprop} together with Corollary \ref{phi-ent to cheeger} show that (iii) $\Rightarrow$ (ii).  The rest of the Theorem, except (v), is a restatement of the results of Section \ref{L_1 entropy inequalities from U-bounds} and the current one. \\
To see (v) we notice that using (\ref{qequivalence}) for small $t>0$ (as well as small $1-t>0$) we have
\[\mathcal{U}_2(t) \leq \bar C_0\mathcal{U}_q(t)
\]
with some $\bar C_0\in(0,\infty)$, and thus there is a constant $\bar C\in(0,\infty)$ such that for all $t\in(0,1)$
\[\mathcal{U}_2(t) \leq \bar C \mathcal{U}_q(t)
\]
Hence, by (iii), we have the following isoperimetric relation
\[\mathcal{U}_2(t)\leq \tilde C \mu^+(A)\]
for any set $A$ with $\mu(A)=t$. Thus, if we are working with Euclidean distance (or we are in finite dimensions when distances given by $l_p$ norms are equivalent), by arguments of \cite{barthemaurey} the $IFI_2$ is true.
\end{proof}

\begin{rem}
We remark that generally perturbation of  $IFI_2$  is a difficult matter if the unbounded log of the density is involved. Our route via $U$-bounds allows us to achieve that very effectively.

Secondly, as conjectured in \cite{B-Z} for $q\in(1,2]$ it would be natural to expect the following
functional isoperimetric inequality with optimal isoperimetric function
\[\mathcal{U}_q(\mu f) \leq \mu \sqrt[q]{\mathcal{U}_q^q(f)+ C_q|\nabla f|_q^q}
 \eqno{(IFI_q)}
\]
with some $C_q\in(0,\infty)$ for all differentiable functions $0\leq f\leq 1$.
One of the motivations for such a relation is that (as shown in \cite{B-Z}) it implies $LS_q$.
Using  $IFI_2$ and the relation of $l_q$ norms, in finite dimension one can see that
\[ 
\mathcal{U}_2(\mu f) \leq \mu \sqrt[q]{\mathcal{U}_2^q(f)+ C_2'|\nabla f|_q^q}.
\]
In the right-hand side, using the asymptotic relation between isoperimetric functions,
one could also replace $\mathcal{U}_2$ with $\mathcal{U}_q$. The question remains if adjusting
the left-hand side in a similar way would still preserve the inequality in the desired sharp form. 
\end{rem}


\section{Application of results}
\label{Sec4.Application of results}

In order to see where these results can be applied, suppose we are still working in the general situation described at the start of this paper, and define a probability measure
 \begin{equation}
 \label{measure}
 d\mu_p := \frac{e^{-\alpha d^p}}{Z}d\lambda
 \end{equation}
on $\R^N$, with $\alpha>0$, $p\in(1, \infty)$ and normalisation constant $Z$.  Recall that here $d:\R^N\times \R^N \to [0, \infty)$ is a metric on $\R^N$.  We have the following result which can be found in \cite{H-Z}.
\begin{proposition}\label{pro4.1}
\label{U bound prop}
Let $\mu_p$ be given by \eqref{measure}.  Suppose that we have
\begin{itemize}
\item[{\rm (i)}] $\frac{1}{\sigma}\leq |\nabla d| \leq1$ almost everywhere for some $\sigma\in[1, \infty)$;
\item[{\rm (ii)}] $\Delta d \leq K + \alpha p\varepsilon d^{p-1}$ on $\{x:d(x)\geq1\}$, for some $K\in[0, \infty), \varepsilon\in[0, \frac{1}{\sigma^2})$.
\end{itemize}
Then there exist constants $A,B\in[0,\infty)$ such that
\[
\mu_p\left(|f|d^{p-1}\right) \leq A\mu_p|\nabla f| + B\mu_p|f|.
\]
\end{proposition}

This proposition gives conditions under which the bound \eqref{U bound again} in Theorem \ref{summary} holds for a particular choice of $U$ and $\beta$.  Indeed, we thus have the following corollary:

\begin{corollary}\label{cor4.2}
\label{conditions}
Let  $\mu_p$ be given by \eqref{measure}.  Suppose that conditions {\rm (i)} and {\rm (ii)} of Proposition \ref{U bound prop} are satisfied.  Suppose also that the measure $d\lambda$ satisfies the classical Sobolev inequality \eqref{class sob} together with the Poincar\'e inequality in balls \eqref{Poincare in ball}. Then inequalities {\rm (i)-(iii)} of Theorem \ref{summary} are satisfied, with $q$ such that $\frac{1}{p} + \frac{1}{q}=1 $.\\
Moreover, if $p\geq 2$  {\rm (iv)-(v)} are also true. \end{corollary}

\begin{proof}
For $U = \alpha d^p$ and $\beta=\frac{1}{q}$ we have
\begin{align*}
\mu_p\left(|f|\left(U^\beta + |\nabla U|\right)\right)  &\leq \mu_p\left(|f|\left(\alpha^\beta d^{\beta p} + \alpha p d^{p-1}\right)\right) \\
& \leq (\alpha^\beta + \alpha p)\mu_p(|f|d^{p-1}).
\end{align*}
Therefore by Proposition \ref{U bound prop}, we have
\[
\mu\left(|f|\left(U^\beta + |\nabla U|\right)\right) \leq \tilde A\mu|\nabla f| +\tilde B \mu|f|
\]
where $\tilde A =  (\alpha^\beta + \alpha p)A$ and $\tilde B =  (\alpha^\beta + \alpha p)B$.  Thus we can apply Theorem \ref{summary}.
\end{proof}

We can perturb the measure in this result and all the inequalities will hold for the perturbed measure, as follows.
\begin{corollary}\label{cor4.3}
Let $d{\hat \mu}  = e^{-W-V}/{\hat Z} d\mu_p$ be the probability measure described in Corollary \ref{perturbation2} with unbounded locally Lipschitz $W$ and bounded measurable $V$.
Then 
${\hat \mu}$ enjoys all properties as $\mu_p$ in Corollary \ref{cor4.2}.
\end{corollary}

\begin{rem}\label{rem4.4}
The conditions of Corollary \ref{conditions} are easily seen to be satisfied in the Euclidean case, when we are dealing with the standard gradient and Laplacian in $\R^N$, and $d(x) = |x|$.  In this situation, with $p=2$, the inequalities we prove are already known (see \cite{Led1}), though the proof we give here is new.
\end{rem}

The value of our results is that they can be used in more general situations than the Euclidean one.  In particular it can be applied in the following setting.

\begin{Example}\label{Ex4.5}
[$H$-type groups]
Let $\mathfrak{g}$ be a (finite-dimensional real) Lie algebra and let $\mathfrak{z}$ denote its centre (i.e. $[\mathfrak{g}, \mathfrak{z}]=0$). We say that $\mathfrak{g}$ is of {\rm H-type} if it admits a vector space decomposition
\[
\mathfrak{g} = \mathfrak{v}\oplus\mathfrak{z}
\]
where $\left[\mathfrak{v}, \mathfrak{v}\right]\subseteq \mathfrak{z}$, such that there exists an inner product $\langle \cdot, \cdot\rangle$ on $\mathfrak{g}$ such that $\mathfrak{z}$ is an orthogonal complement to $\mathfrak{v}$, and the map $J_Z: \mathfrak{v}\mapsto\mathfrak{v}$ given by
\[
\langle J_ZX, Y\rangle = \langle [X, Y], Z\rangle
\]
for $X, Y\in \mathfrak{v}$ and $Z\in\mathfrak{z}$ satisfies $J^2_Z = - |Z|^2I$ for each $Z\in \mathfrak{z}$.  An {\rm H-type group} is a simply connected Lie group $\mathbb{G}$ whose Lie algebra is of H-type.

Such a group is a Carnot group of step 2 (see \cite{B-L-U} for details).  In particular the Heisenberg group is an H-type group with a one-dimensional centre.  However, there also exist H-type groups with centre of any dimension.
On an H-type group $\mathbb{G}$ we consider vector fields $X_1, \dots, X_m$ which form an orthonormal basis of $\mathfrak{v}$.  The sub-Laplacian (or Kohn operator) is given by $\Delta_\mathbb{G} := \sum_{i=1}^m X_i^2$ and sub-gradient by $\nabla_\mathbb{G} := (X_1,\dots, X_m)$.  The associated Carnot-Carath\'eodory distance is defined by
\[
d(x, y) := \sup\{ f(x) - f(y): f\ {\rm such\ that}\ |\nabla_\mathbb{G} f| \leq 1\}.
\]
It is shown in \cite {H-Z} that conditions {\rm (i)} and {\rm (ii)} of Proposition \ref{U bound prop} are satisfied in this setting.  Moreover, the Lebesgue measure $d\lambda$ satisfies the classical Sobolev inequality \eqref{class sob} and Poincar\'e inequality in balls \eqref{Poincare in ball} with the sub-gradient $\nabla_\mathbb{G}$ (see \cite{Var}).  Thus, by Corollary \ref{conditions} we arrive at the following:
\begin{theorem}\label{thm4.6}
Let $\mathbb{G} = (\mathbb{R}^{m+n}, \circ)$ be an H-type group, equipped with Carnot-Carath\'eodory distance $d$ and canonical sub-gradient $\nabla_\mathbb{G}$ as described above.  Let 
\[
d\mu_p := \frac{e^{-\alpha d^p}}{Z}d\lambda
\]
with $p>1$ and $\alpha>0$ be a probability measure on $\G$
and
\[d{\hat \mu}  = e^{-W-V}/{\hat Z} d\mu_p
\]
with $W\equiv W(d)$ satisfying conditions as  in Corollary \ref{cor4.2} with horizontal gradient 
and $V$ a bounded measurable function.
Then inequalities {\rm (i)-(iii)} of Theorem \eqref{summary} are satisfied with $q$ such that $\frac{1}{p}+\frac{1}{q}=1$.
Moreover for $p\geq 2$, the measure $\hat \mu$ satisfies $LS_q$ and $IFI_2$.
\end{theorem}
\end{Example}

\noindent \textbf{U-Bounds versus Gradient Bounds for Heat Kernel.}\label{A.GradBounds} As a conclusion to this section we mention that our setup is naturally inclusive for the following  gradient bounds for the heat kernel on the H-type groups which has recently attracted considerable attention (see e.g. \cite{BBBC, DM, El1, El2, HQL} and references therein).  

Indeed, in the following let $\mathbb{G}$ be an H-type group.

\begin{corollary}
\label{cor4.7}
The semigroup  $P_t\equiv e^{t\Delta_\mathbb{G}}$ satisfies the following 
\[
|\nabla_\mathbb{G}  P_t f|\leq C_1(t)P_t |\nabla_\mathbb{G}  f|\] 
with $C_1(t)\in(0,\infty)$ independent of $f$.
\end{corollary}
Due to the group covariance, it is sufficient to show the bound at the identity element
and, thanks to an action of the dilations, one only needs to establish it at $t=1$. Denoting
the corresponding heat kernel by $h$, we see that a bound of the following quantity is necessary
\[ \left\vert\int f\nabla_\mathbb{G} h d\lambda\right\vert =  \left\vert\int (f - <f>)\nabla_\mathbb{G} h d\lambda\right\vert  \leq \int |f - <f>| \cdot |\nabla_\mathbb{G} \log h| h d\lambda 
\]
with $<f>\equiv \int f h d\lambda$.
If one has a bound of the form
\[
|\nabla_\mathbb{G} \log h| \leq V(d)
\]
with a function $V$ growing to infinity and for which the following $U$-bound is satisfied
\[
\int |f - <f>| \cdot V(d) h d\lambda \leq C  \int |\nabla_\mathbb{G} f  |  h d\lambda + D \int |f - <f>|  h d\lambda
\]
with some $C,D\in[0,\infty)$ independent of $f$, then  -- as we have argued in the previous sections -- 
one can show the following Cheeger type bound
\[\int |f - <f>|  h d\lambda \leq \alpha  \int |\nabla_\mathbb{G} f  |  h d\lambda.
\]  
Consequently we arrive at
\[ \left\vert\int f\nabla_\mathbb{G} h d\lambda\right\vert\leq (C+D\alpha) \int |\nabla_\mathbb{G} f  |  h d\lambda.
\]
Thanks to the following heat kernel bounds of \cite{El1} (see also \cite{HQL} and \cite{B-G-G})
\begin{align}\label{hkb1}
 h(x,z)  &\asymp \frac{1 +(d(0, (x,z)))^{2n-m-1}}{1 +(|x|d(0, (x,z)))^{n-1/2}}
e^{-\frac14 d(0,(x,z))^2}\\
\label{hkb2}
|\nabla \log h(x,z)| &\leq  C (1+ d(0, (x,z))) ,  \  
\end{align}
we see that this strategy can be realised positively.
While the gradient bounds still remain a challenge for more complicated groups,
it may be useful to keep this observation in mind, as in principle it allows for a heat kernel bound \eqref{hkb1} with far less precise description of the slowly varying factor, (provided the corresponding control distance $d$ satisfies a sufficiently good Laplacian bound outside some compact set).
\\

\noindent \textbf{U-Bounds versus Integrated Gaussian Bounds for Heat Kernel.}\label{B.GaussBounds}

\noindent Assuming a bound of the following form
\begin{equation}
\label{assume bound}
\mu (f d) \leq C \mu |\nabla f| + D \mu (f),
\end{equation}
for a function $f= e^{\lambda \min(d,L)^2}$, we get
\begin{align*}\mu \left(e^{\lambda \min(d,L)^2} \min(d,L)\right) &\leq 2\lambda C \mu \left(e^{\lambda \min(d,L)^2}\min(d,L)|\nabla \min(d,L)|\right) \\
& \quad + D \mu \left(e^{\lambda \min(d,L)^2}\right) \\
&\leq 2\lambda C  \mu \left(e^{\lambda \min(d,L)^2}\min(d,L)\right)  + D \mu \left(e^{\lambda \min(d,L)^2}\right).
\end{align*}
If $2\lambda C < 1$, this implies 
\begin{equation}\label{B.1}
\mu \left(e^{\lambda \min(d,L)^2} \min(d,L)\right) \leq D' \mu \left(e^{\lambda \min(d,L)^2}\right)
\end{equation}
with $D'\equiv D(1-2\lambda C)^{-1}$.
Next, choosing $f= e^{\lambda \min(d,L)^2}\min(d,L)$ instead in \eqref{assume bound}, we obtain
\begin{align*}
\mu \left(e^{\lambda \min(d,L)^2} \min(d,L)^2\right) &\leq C  \mu \left|\nabla \left(e^{\lambda \min(d,L)^2} \min(d,L)\right)\right| + D \mu \left(e^{\lambda \min(d,L)^2} \min(d,L)\right) \\
&\leq 2\lambda C  \mu \left(e^{\lambda \min(d,L)^2} \min(d,L)^2\right) + D  \mu  \left(e^{\lambda \min(d,L)^2} \min(d,L)\right)\\
& \quad + C \mu \left(e^{\lambda \min(d,L)^2 }\right).
\end{align*} 
Thus using \eqref{B.1}, we obtain 
\begin{align*}\mu \left(e^{\lambda \min(d,L)^2} \min(d,L)^2\right) &\leq 2\lambda C  \mu \left(e^{\lambda \min(d,L)^2} \min(d,L)^2\right)
& + (D'+C) \mu \left(e^{\lambda \min(d,L)^2 }\right).\\
\end{align*}
Rearranging this, for $2\lambda C\leq 2\lambda_0C<1$,
\[ \frac{d}{d\lambda} \mu \left(e^{\lambda \min(d,L)^2}\right) = \mu \left(e^{\lambda \min(d,L)^2} \min(d,L)^2\right) \leq D'' \mu \left(e^{\lambda \min(d,L)^2 }\right)
\]
with $D''\equiv (D'+C)(1-2\lambda_0C)^{-1}$. Solving this differential inequality
and passing with $L\to\infty$, we arrive at the following:
\begin{theorem}\label{thm4.8}
$($\textbf{\!\! Integrated Gaussian Bound }$)$\\
Suppose the following is true
\[ \mu (f d) \leq C \mu |\nabla f| + D \mu (f)
\]
with some constants $C,D\in(0,\infty)$.
Then
\[ \mu \left(e^{\lambda d^2}\right) \leq e^{\lambda D''}
\]
for $2\lambda C\leq 2\lambda_0C<1$ with some constant $D''\in(0,\infty)$.  
\end{theorem}
$\ $\\
\noindent
See Appendix 1 for some generalisation of this idea. \\

\noindent \textbf{From Gradient Bounds for Heat Kernel to U-Bounds.}\label{C.GradBds4HK2UBds}
From the point of view of the computations of \cite{H-Z}  we start with
\[ h\nabla f = \nabla(fh) - f\nabla h 
\]
and, with a unitary linear functional $\boldsymbol{\alpha}$, we get
\[ \int \boldsymbol{\alpha}(\nabla f) hd\lambda = \int \boldsymbol{\alpha}\left(\nabla (fh)\right) d\lambda
+ \int f \boldsymbol{\alpha}(\nabla \log \frac1h) h d\lambda.
\]
Hence, one gets
\[ \int f \left(\boldsymbol{\alpha}(\nabla \log \frac1h)   - div \boldsymbol{\alpha}\right) hd\lambda
\leq \int |\nabla f| \cdot |\boldsymbol{\alpha}| hd\lambda.
\]
If the expression in the bracket on the left-hand side can be shown to have a treatable bound from below, such a bound can be a useful source of analysis.\\


\section{Extension to infinite dimensions}
\label{Sec5.Extension to infinite dimensions}
In this section we aim to extend the $L_1\Phi$-entropy inequality to the infinite dimensional setting, where we include some bounded interactions.  The setup will be as follows.
~

\paragraph{\textit{The Spin Space:}} Let $\mathcal{M} = (\R^N, d)$ be a metric space equipped with Lebesgue measure $d\lambda$, general sub-gradient $\nabla = (X_1, \dots, X_m)$ consisting of divergence free (possibly non-commuting) vector fields and sub-Laplacian $\Delta := \sum_{i=1}^mX_i^2$, as above.

\paragraph{\textit{The Lattice:}} Let $\Z^D$ be the $D$-dimensional lattice for some fixed $D\in\mathbb{N}$, equipped with the lattice metric $dist(\cdot, \cdot)$ defined by
\[
dist({\bf i},{\bf j}) := \sum_{l=1}^D|i_l - j_l|
\]
for ${\bf i} = (i_1, \dots, i_D), {\bf j} = (j_1, \dots, j_D)\in\Z^D$.  For ${\bf i},{\bf j}\in\Z^D$ we will also write 
\[
{\bf i} \sim {\bf j} \qquad \Leftrightarrow \qquad dist({\bf i},{\bf j}) = 1
\]
i.e. ${\bf i}\sim {\bf j}$ when ${\bf i}$ and ${\bf j}$ are nearest neighbours in the lattice.  For $\Lambda \subset \mathbb{Z}^D$, we will write $\Lambda^c\equiv  \mathbb{Z}^D\setminus\Lambda$, $|\Lambda|$ for the cardinality of $\Lambda$, and $\Lambda \subset \subset \mathbb{Z}^D$ when $|\Lambda|<\infty$.

\paragraph{\textit{The Configuration Space:}}
Let $\Omega := \left(\mathcal{M}\right)^{\Z^D}$ be the \textit{configuration space}.  Given $\Lambda \subset \mathbb{Z}^D$ and $\omega = (\omega_{\bf i})_{{\bf i}\in\mathbb{Z}^D}\in\Omega$, let $\omega_\Lambda := (\omega_{\bf i})_{{\bf i}\in\Lambda} \in\left(\mathcal{M}\right)^\Lambda$ (so that $\omega\mapsto \omega_\Lambda$ is the natural projection of $\Omega$ onto $\mathcal{M}^\Lambda$).

Given $\omega\in\Omega$ we introduce the injection: $\mathcal{M}^\Lambda \to \Omega$, defined by $\eta \in \mathcal{M}^\Lambda \mapsto \eta\bullet_\Lambda\omega$ where $(\eta\bullet_\Lambda\omega)_{\bf i
} = \eta_{\bf i}$ when ${\bf i }\in\Lambda$ and $(\eta\bullet_\Lambda\omega)_{\bf i} = \omega_{\bf i}$ when ${\bf i}\in\Lambda^c$. 

Let $f\colon \Omega \to \mathbb{R}$.  Then for ${\bf i}\in\mathbb{Z}^D$ and $\omega \in \Omega$ define $f_{\bf i}(\cdot|\omega)\colon\mathcal{M} \to \mathbb{R}$ by
\[
f_{\bf i}(x|\omega) := f(x\bullet_{\{{\bf i}\}}\omega).
\]
Let $C^{(n)}(\Omega)$, $n\in\mathbb{N}$ denote the set of all functions $f$ for which we have $f_{\bf i}(\cdot|\omega) \in C^{(n)}(\mathcal{M})$ for all ${\bf i}\in\mathbb{Z}^D$ .  For ${\bf i}\in\mathbb{Z}^D, k\in\{1,\dots, m\}$ and $f\in C^{(1)}(\Omega)$, define
$$X_{{\bf i},k}f(\omega) := X_kf_{\bf i}(x|\omega)\vert_{x = \omega_{\bf i}},$$
where $X_1,\dots,X_m$ are the vector fields on $\mathcal{M}$.

Define similarly $\nabla_{\bf i}f(\omega) := \nabla f_{\bf i}(x|\omega)\vert_{x = \omega_{\bf i}}$ and $\Delta_{\bf i}f(\omega) := \Delta f_{\bf i}(x|\omega)\vert_{x = \omega_{\bf i}}$ for suitable $f$, where $\nabla$ and $\Delta$ are the sub-gradient and the sub-Laplacian on $\mathcal{M}$ respectively.  For $\Lambda \subset \Z^D$, set $\nabla_\Lambda f = (\nabla_{\bf i}f)_{{\bf i}\in\Lambda}$ and
$$|\nabla_\Lambda f| := \sum_{{\bf i}\in\Lambda}|\nabla_{\bf i}f|.$$

Finally, a function $f$ on $\Omega$ is said to be \textit{localised} in a set $\Lambda\subset\Z^D$ if $f$ is only a function of those coordinates in $\Lambda$.

\paragraph{\textit{Local Specification and Gibbs Measure:}}
Let $\Psi = (\psi_{X})_{X\subset\subset\Z^D}$ be a family of $C^2$ functions such that $\psi_{X}$ is localised in $X\subset\subset\Z^D$.  Assume that $\psi_X\equiv 0$ whenever the diameter of $X$ is greater than positive constant $R$. We will also assume that there exists a constant $M\in(0,\infty)$ such that $\|\psi_{X}\|_\infty \leq M$ and $\|\nabla_{\bf i}\psi_{X}\|_\infty\leq M$ for all ${\bf i}\in \Z^D$.  We say $\Psi$ is a bounded potential of range $R$.  For $\omega\in\Omega$, define
$$H_\Lambda^\omega(x_\Lambda) = \sum_{\Lambda\cap X\neq\emptyset}\psi_{X}\left(x_\Lambda\bullet_\Lambda\omega\right),$$
for $x_\Lambda = (x_{\bf i})_{{\bf i}\in\Lambda} \in \mathcal{M}^\Lambda$.

Let $U$ be a locally Lipschitz function on $\mathcal{M}$ which is bounded from below and such that $\int_\mathcal{M}e^{-U}d\lambda <\infty$.  Suppose also that $\forall L\geq0$ there exists $r=r(L)$ such that
\[
\{U\leq L\} \subset B(r).
\]

Let $d\mu = \frac{e^{-U}}{Z}d\lambda$, so that $\mu$ is a probability measure on $\mathcal{M}$, and let
\[
\mu_\Lambda(dx_\Lambda) := \otimes_{{\bf i}\in\Lambda}\mu(dx_{\bf i})
\]
be the product measure on $\mathcal{M}^\Lambda$.
Now define
\begin{equation}
\label{local spec}
\E^{\omega}_\Lambda(dx_\Lambda)=\frac{e^{JH^{\omega}_\Lambda(x_\Lambda)}}{\int e^{JH^{\omega}_\Lambda(x_\Lambda)}\mu_\Lambda (dx_\Lambda)} \mu_\Lambda(dx_{\Lambda})\equiv\frac{e^{JH^{\omega}_\Lambda(x_\Lambda)}}{Z^\omega_\Lambda} \mu_\Lambda(dx_{\Lambda})
\end{equation}
for $J \in\R$.  We will write $\mu_{\{{\bf i}\}} = \mu_{\bf i}$ and $\E^{\omega}_{\{{\bf i}\}} = \E^{\omega}_{\bf i}$ for ${\bf i} \in \Z^D$.  We finally define an infinite volume \textit{Gibbs measure} $\nu$ on $\Omega$ to be a solution of the (DLR) equation:
\begin{equation}
\label{DLR}
\nu \E^{\cdot}_\Lambda f = \nu f
\end{equation}
for all bounded measurable functions $f$ on $\Omega$.  $\nu$ is a measure on $\Omega$ which has $\E^{\omega}_\Lambda$ as its finite volume conditional measures.

~

Following for example \cite{G-Z}, \cite{I-P}, the extension of Theorem \ref{main} to this infinite dimensional setting will take the following form.

\begin{theorem}\label{5.1}
\label{beta LS for Gibbs thm}
Suppose that the classical Sobolev inequality \eqref{class sob} and that the Poincar\'e inequality in balls \eqref{Poincare in ball} are both satisfied.  Suppose also that inequality \eqref{U-bound} is satisfied, i.e. there exist constants $A, B\in(0,\infty)$ such that
\[
\mu\left(|f|\left(|U|^\beta + |\nabla U|\right)\right) \leq A\mu|\nabla f| + B\mu|f|
\]
for some $\beta\in(0,1]$ and locally Lipschitz functions $f:\mathcal{M}\to\R$.  Then there exists $J_0>0$ such that for $|J|<J_0$, the Gibbs measure $\nu$ is unique and there exists a constant $C$ such that
\begin{equation}
\label{beta LS for GIbbs}
\mathbf{Ent}_\nu^\Phi(|f|) \leq C\nu\left(\sum_{{\bf i}\in\Z^D}|\nabla_{\bf i} f|\right),
\end{equation}
where $\Phi(x) = x\left(\log(1+x)\right)^\beta$, for all $f$ for which the right-hand side is well defined.
\end{theorem}

For notational simplicity, we will only prove Theorem \ref{beta LS for Gibbs thm} in the case $R=1$ and $D=2$, but the method can easily be extended to general $R$ and $D$, (see e.g. \cite{G-Z} for the idea of the general scheme).

Define the sets
\begin{align*}
&\Gamma_0=(0,0)\cup\{{\bf j} \in \mathbb{Z}^2 : dist({\bf j},(0,0))=2n \text{\; for some \;}n\in\mathbb{N}\}, \\  
&\Gamma_1=\mathbb{Z}^2\smallsetminus\Gamma_0 .
\end{align*}
Note that $dist({\bf i},{\bf j})>1$ for all ${\bf i,j} \in\Gamma_k,k=0,1$ and $\Gamma_0\cap\Gamma_1=\emptyset$.  Moreover $\mathbb{Z}^2=\Gamma_0\cup\Gamma_1$.  For the sake of notation, we will write $\mathbb{E}_{\Gamma_k}=\mathbb{E}_{\Gamma_k}^{\omega}$ for $k=0,1$. We will also define 
$$ \mathcal{P}:=\mathbb{E}_{\Gamma_1}\mathbb{E}_{\Gamma_{0}}.$$

The proof will rely on the following few Lemmata.

\begin{lemma}\label{lem5.2}
\label{single site lem} 
Under the conditions of Theorem \ref{beta LS for Gibbs thm}, there exist constants $\hat{c}_0$ and $\hat{c}$ independent of ${\bf i}\in\Z^D$ and $\omega\in\Omega$ such that
\begin{equation}
\label{cheeger for single site}
\E_{\bf i}^\omega\left|f - \E_{\bf i}^\omega f\right| \leq \hat{c}_0\E_{\bf i}^\omega|\nabla_{\bf i} f|
\end{equation}
and
\begin{equation}
\label{beta LS for single site}
\mathbf{Ent}_{\E_{\bf i}^\omega}^\Phi(|f|) \leq \hat{c}\E_{\bf i}^\omega|\nabla_{\bf i} f|
\end{equation}
for all ${\bf i}\in\Z^D$ and $\omega\in\Omega$.
\end{lemma}

\begin{proof}
Firstly, by Theorem \ref{cheeger from U-bound}, we have that there exists a constant $c_0$ independent of ${\bf i}$ such that
\[
\mu_{\bf i}\left|f-\mu_{\bf i}f| \leq c_0\mu_{\bf i}|\nabla_{\bf i} f\right|.
\]
Since
\[
osc\left(H_{\bf i}^\omega\right) \leq 2\|H_{\bf i}^\omega\|_\infty \leq 2\sum_{\{{\bf i}\}\cap X \neq \emptyset}\|\psi_X\|_\infty\leq 8M.
\]
by a standard result about bounded perturbations of Poincar\'e type inequalities (see \cite{B-Z}), inequality \eqref{cheeger for single site} holds.

Moreover, by the assumptions and Theorem \ref{main}, we have
\[
\Ent_{\mu_{\bf i}}^\Phi(|f|) = \mu_{\bf i}(\Phi(|f|)) - \Phi(\mu_{\bf i} |f|) \leq c \mu_{\bf i}|\nabla_{\bf i} f|
\]
for all ${\bf i} \in \Z^D$.  Thus by the bounded perturbation Corollary \ref{perturbation2}, \eqref{beta LS for single site} holds.
\end{proof}

\begin{lemma}\label{lem5.3}
\label{gradient terms}
Under the conditions of Theorem \ref{beta LS for Gibbs thm}, there exists $J_0>0$ such that for $|J|<J_0$, there exists a constant and $\varepsilon\in(0,1)$ such that
\[
\nu\left|\nabla_{\Gamma_k}\left(\E_{\Gamma_l}f\right)\right| \leq \nu\left|\nabla_{\Gamma_k} f\right| + \varepsilon\nu\left|\nabla_{\Gamma_1}f\right|
\]
for $k,l\in\{0,1\}$ such that $k\neq l$.
\end{lemma}

\begin{proof}
We suppose $k=1$ and $l=0$.  The case $k=0, l=1$ follows similarly.  We can write
\begin{align*}
\nu\left|\nabla_{\Gamma_1}\left(\E_{\Gamma_0}f\right)\right| &= \nu\left(\sum_{{\bf i}\in\Gamma_1}\left|\nabla_{\bf i}\left(\E_{\Gamma_0}f\right)\right|\right) \leq \nu\left(\sum_{{\bf i}\in\Gamma_1}\left|\nabla_{\bf i}\left(\E_{\{\sim{\bf i}\}}f\right)\right|\right)\\
&\leq \nu\sum_{{\bf i}\in\Gamma_1}\left|\nabla_{\bf i} f\right|  + |J|\nu\left(\sum_{{\bf i}\in\Gamma_1}\left|\E_{\{\sim{\bf i}\}}\left(f\left[\nabla_{\bf i}H_{\{\sim{\bf i}\}} - \E_{\{\sim{\bf i}\}}\nabla_{\bf i}H_{\{\sim{\bf i}\}}\right]\right)\right|\right)
\end{align*}
where we have used \eqref{DLR} and denoted $\{\sim{\bf i}\} = \{{\bf j}: {\bf j} \sim {\bf i}\}$.  Now set $\mathcal{W}_{\bf i} = W_{\bf i} - \E_{\{\sim{\bf i}\}}W_{\bf i}$, where $W_{\bf i} = \nabla_{\bf i}H_{\{\sim{\bf i}\}}^\omega$.  Then since $\E_{\{\sim{\bf i}\}}\mathcal{W}_{\bf i} =0$, we have that
\begin{equation}
\label{covariance}
\nu\left|\nabla_{\Gamma_1}\left(\E_{\Gamma_0}f\right)\right| \leq \nu\sum_{{\bf i}\in\Gamma_1}\left|\nabla_{\bf i} f\right|  + |J|\nu\left(\sum_{{\bf i}\in\Gamma_1}\left|\E_{\{\sim{\bf i}\}}\left(f-\E_{\{\sim{\bf i}\}}f\right)\mathcal{W}_{\bf i}\right|\right)
\end{equation}

Now, by our assumptions on the potential, we have $\|\mathcal{W}_{\bf i}\|_{\infty} \leq 8M$ for all ${\bf i}\in\Z^D$, so that
\begin{equation}
\label{bounded first}
\left|\E_{\{\sim{\bf i}\}}\left(f-\E_{\{\sim{\bf i}\}}f\right)\mathcal{W}_{\bf i}\right| \leq 8M\E_{\{\sim{\bf i}\}}\left|f-\E_{\{\sim{\bf i}\}}f\right|.
\end{equation}
Note that by construction, $\E_{\{\sim{\bf i}\}}$ is a product measure.  Now by Lemma \ref{single site lem} together with Lemma \ref{tensorisation} there exists a constant $\hat{c}_0$ such that
\begin{equation}
\label{cheeger for spec}
\E_{\{\sim{\bf i}\}}\left|f-\E_{\{\sim{\bf i}\}}f\right| \leq \hat{c}_0 \E_{\{\sim{\bf i}\}}|\nabla_{\{\sim {\bf i}\}} f|.
\end{equation}
Using \eqref{bounded first} and \eqref{cheeger for spec} in \eqref{covariance}, we then arrive at
\begin{align*}
\nu\left|\nabla_{\Gamma_1}\left(\E_{\Gamma_0}f\right)\right| &\leq \nu\sum_{{\bf i}\in\Gamma_1}\left|\nabla_{\bf i} f\right| + 8M\hat{c}_0|J|\nu\left(\sum_{{\bf i}\in\Gamma_1}|\nabla_{\{\sim {\bf i}\}} f|\right)\\
&= \nu\sum_{{\bf i}\in\Gamma_1}\left|\nabla_{\bf i} f\right| + 32M\hat{c}_0|J|\nu\left(\sum_{{\bf i}\in\Gamma_0}|\nabla_{\bf i} f|\right).
\end{align*}
Thus taking $J_0 = \frac{1}{32M\hat{c}_0}$ proves the lemma.
\end{proof} 

\begin{lemma}\label{lem5.4}
\label{convergence}
Under the conditions of Theorem \ref{beta LS for Gibbs thm}, there exists $J_0>0$ (given by Lemma \ref{gradient terms}) such that for $|J|<J_0$, $\mathcal{P}^rf$ converges almost everywhere to $\nu f$, where we recall that $\mathcal{P} = \E_{\Gamma_1}\E_{\Gamma_0}$.  In particular $\nu$ is unique.
\end{lemma}

\begin{proof}
The proof is standard: see for example Lemma 5.6 of \cite{I-P}.
\end{proof}

\begin{proof}[Proof of Theorem \ref{beta LS for Gibbs thm}]
We may suppose $f\geq0$.  Using \eqref{DLR}, write

\begin{align*}
\nu(\Phi(f)) - \Phi( \nu f) &= \nu\E_{\Gamma_0}\left(\Phi(f)\right) - \nu\left(\Phi\left(\E_{\Gamma_0}f\right)\right) \\
& \quad + \nu\left(\Phi\left(\E_{\Gamma_0}f\right)\right) -  \Phi( \nu f) \\
&= \nu\left(\Ent_{\E_{\Gamma_0}}^\Phi(f)\right) + \nu\left( \Ent^\Phi_{\E_{\Gamma_1}}(\E_{\Gamma_0}f)\right)\\
& \quad + \nu\left(\Phi\left(\E_{\Gamma_1}\E_{\Gamma_0}f\right)\right) - \Phi( \nu f).
\end{align*}
Since probability measures $\E_{\Gamma_0}$ and $\E_{\Gamma_1}$ are product measures by construction, we have by Lemmas \ref{tensorisation} and \ref{single site lem} that they both satisfy $L_1\Phi$-entropy inequalities with constant $\hat{c}$.  Therefore, the above yields
\begin{align*}
\nu(\Phi(f)) - \Phi( \nu f) &\leq \hat{c}\nu|\nabla_{\Gamma_0} f| + \hat{c}\nu\left|\nabla_{\Gamma_1}\left(\E_{\Gamma_0}f\right)\right| \\
& \quad + \nu\left(\Phi(\mathcal{P}f)\right) - \Phi(\nu f).
\end{align*}
We can similarly write
\begin{align*}
\mu\left(\Phi(\mathcal{P}f)\right) &= \nu\left(\Ent_{\E_{\Gamma_0}}^\Phi(\mathcal{P}f)\right) + \nu\left( \Ent^\Phi_{\E_{\Gamma_1}}(\E_{\Gamma_0}\mathcal{P}f)\right)+ \nu\left(\Phi\left(\mathcal{P}^2f\right)\right)\\
& \leq \hat{c}\nu\left|\nabla_{\Gamma_0}\mathcal{P}f\right| + \hat{c}\nu\left|\nabla_{\Gamma_1}\left(\E_{\Gamma_0} f\right)\right| +  \nu\left(\Phi\left(\mathcal{P}^2f\right)\right).
\end{align*}
Repeating this process, after $r$ steps we see that
\begin{align}
\label{r steps}
\nu(\Phi(f)) - \Phi( \nu f) &\leq \hat{c}\sum_{k=0}^{r-1}\nu\left|\nabla_{\Gamma_0}\mathcal{P}^kf\right| + \hat{c}\sum_{k=0}^{r-1}\nu\left|\nabla_{\Gamma_1}\left(\E_{\Gamma_0}\mathcal{P}^kf\right)\right| \\
&\quad + \nu\left(\Phi\left(\mathcal{P}^rf\right)\right) - \Phi(\nu f).
\end{align}
We may control the first and second terms using Lemma \ref{gradient terms}.  Indeed
\begin{align}
\label{first term}
\nu\left|\nabla_{\Gamma_0}\mathcal{P}^kf\right| &\leq \varepsilon^2\nu\left|\nabla_{\Gamma_0} \mathcal{P}^{k-1}f\right|\nonumber\\
&\leq \varepsilon^{2k-1}\nu\left|\nabla_{\Gamma_1} \E_{\Gamma_0}f\right|\nonumber\\
&\leq \varepsilon^{2k-1}\nu\left|\nabla_{\Gamma_1} f\right| + \varepsilon^{2k}\nu\left|\nabla_{\Gamma_0}f\right|.
\end{align}
Similarly
\begin{align}
\label{second term}
\nu\left|\nabla_{\Gamma_1}\left(\E_{\Gamma_0}\mathcal{P}^kf\right)\right| \leq \varepsilon^{2k}\nu\left|\nabla_{\Gamma_1}f\right| + \varepsilon^{2k+1}\nu\left|\nabla_{\Gamma_0}f\right|.
\end{align}
Using \eqref{first term} and \eqref{second term} in \ref{r steps} yields
\begin{align*}
\nu(\Phi(f)) - \Phi( \nu f) &\leq \hat{c}\left(1+\varepsilon^{-1}\right)\left[\left(\sum_{k=0}^{r-1}\varepsilon^{2k}\right)\nu\left|\nabla_{\Gamma_1}f\right| +\left(\sum_{k=0}^{r-1}\varepsilon^{2k+1}\right)\nu\left|\nabla_{\Gamma_0}f\right|\right] \\
&\quad + \nu\left(\Phi\left(\mathcal{P}^rf\right)\right) - \Phi(\nu f).
\end{align*}
By Lemma \ref{convergence} we have that $\lim_{r\to\infty}\mathcal{P}^rf = \nu f$, $\nu$-almost surely.  Therefore taking the limit as $r\to\infty$ in the above (which exists since $\varepsilon\in(0,1)$) yields
\[
\nu(\Phi(f)) - \Phi( \nu f) \leq C\nu\left|\nabla_{\Z^D} f\right|
\]
where $C=\hat{c}\frac{1+\varepsilon^{-1}}{1-\varepsilon^2}$.
\end{proof}
\smallskip
\noindent Next, we consider $IFI_2$ for a family of examples. In particular we restrict ourselves to a situation when $\mathcal{M}$ is an H-type group
and assume that for ${\bf i} \in \Z^D$
\begin{equation}\label{i.1}
U_{\bf i}\equiv \sum_{k=0,..,p-1}\alpha_k d^{p-k}_{\bf i}  \equiv \sum_{k=0,..,p-1}\alpha_k d^{p-k}(\omega_{\bf i}) 
\end{equation}
with $d(\cdot)$ denoting the Carnot-Caratheodory distance from the unit element,
$p\geq 2$, where  $\alpha_0\in(0,\infty)$ and $\alpha_k\in\mathbb{R}$. As above we consider an interaction
\begin{equation}\label{i.2}
H_\Lambda^\omega(x_\Lambda) = \sum_{\Lambda\cap X\neq\emptyset}\psi_{X}\left(x_\Lambda\bullet_\Lambda\omega\right),
\end{equation}
which is assumed to be bounded with bounded (sub-) gradient and  for simplicity is of finite range,
as specified at the beginning of the current section. Moreover we are given a family of regular conditional expectations defined by \eqref{5.1}.  Combining the previous results with those of this section the previous we arrive at the following theorem.

\begin{theorem}\label{5.2} \label{betaLS+IFI_2for Gibbs} $\,$\\
Suppose $p\geq 2$. Then there exists $J_0>0$ such that for $|J|<J_0$ the unique Gibbs measure $\nu$ corresponding to the interaction
$(\ref{i.1})-(\ref{i.2})$ satisfies the following inequalities
\begin{itemize}
 \item[{\rm(i)}]
\begin{equation} \label{beta_LS for Gibbs.bis}
\mathbf{Ent}_\nu^\Phi(|f|) \leq C_1\nu\left(\sum_{{\bf i}\in\Z^D}|\nabla_{\bf i} f|\right),
\end{equation}
where $\Phi(x) = x\left(\log(1+x)\right)^\frac1q$, $\frac1q+\frac1p=1$, with some constant $C_1\in(0,\infty)$, for any $f$ for which the right-hand side is well defined;
\item[{\rm(ii)}] 
\begin{equation} \label{IFI_2.bis}
\mathcal{U}_2(\nu f) \leq  \nu \left( \mathcal{U}_2(f)^2 + C_2\sum_{{\bf i}\in\Z^D}|\nabla_{\bf i} f|^2 
\right)^\frac12
\end{equation}
where $\mathcal{U}_2$ is the Gaussian isoperimetric profile function (as defined in section \ref{Sec3.Isoperimetric Inequalities}), with some constant $C_2\in(0,\infty)$ for any function $0\leq f\leq 1$ for which the right hand side is well defined.
\end{itemize}

\end{theorem}

\begin{proof}
To begin we notice that the reference measure $d\mu$ satisfies a $U$-bound, and therefore the conditional
expectation, (as a perturbation of the reference measure by strictly bounded and strictly positive density), also satisfies the following inequality 
\begin{equation} \label{5.2.1}
\int_{\mathbb{H}} f \left|U\right|^{\frac1q}  d\E_{\bf i}\leq A
\int_{\mathbb{H}} |\nabla_{\bf i} f|  d\E_{\bf i} + B  \int_{\mathbb{H}}  f d\E_{\bf i}
\end{equation}
with some constants $A,B\in(0,\infty)$ independent of ${\bf i}$ and $\omega_{\bf j}$, where $\E_{\bf i}$
denotes the corresponding conditional expectations. Thus we can apply Theorem \ref{thm3.4}
to conclude that the $\E_{\bf i}$'s satisfy Cheeger's inequality, as well as $L_1 \Phi$-entropy and $IFI_2$ bounds with constants independent of ${\bf i}$ and $\omega_{\bf j}$'s.
With this bound the proof of (ii) follows via strategy developed in \cite{Z}.
\end{proof}

\begin{rem}
We remark that once the conditional measures satisfy $L_1 \Phi$-entropy or $IFI_2$ inequalities with constants independent of external conditions, one can show that the Gibbs measure also satisfies $IFI_2$ even when the interactions $H_{\bf i}$ contain an unbounded component, provided we have Cheeger's inequality and appropriate $U$-bounds .
In particular one obtains the following generalisation of the results of \cite{Z} where only the bounded interaction case was studied.

\begin{theorem}\label{5.3}
Suppose $\mathcal{M}\equiv{\mathbb R}$, $U$ is a semibounded polynomial of degree at least $2$ and let
\[
H_{\bf i}^\omega(x_{\bf i}) \equiv \varepsilon\sum_{\{{\bf i}\}\cap X\neq\emptyset}\psi_{X}\left(x_{\bf i}\bullet_{\bf i}\omega\right) +
\varepsilon\sum_{{\bf j}} G_{{\bf i}{\bf j}}  x_{\bf i}\omega_{\bf j}
\] 
with $\psi_{X}$ satisfying conditions of Theorem \ref{5.2}, $\sum_{{\bf j}} |G_{{\bf i}{\bf j}}|<\infty$
and $\varepsilon\in(0,\infty)$.
Then, if $\varepsilon\in(0,\infty)$ is sufficiently small, the corresponding Gibbs measure satisfies $IFI_2$.
\end{theorem}
 
\end{rem}

\begin{rem}
For cylinder functions dependent on $N$ coordinates, adapting the length of the gradient in part {\rm (i)} of Theorem \ref{betaLS+IFI_2for Gibbs}, we get
\begin{equation} \label{beta_LS for Gibbs.bisprime}
\mathbf{Ent}_\nu^\Phi(|f|) \leq C_1\sqrt{N}\nu\left(\sum_{{\bf i}_l\in\Z^D, l=1,..,N}|\nabla_{{\bf i}_l} f|^2\right)^\frac12.
\end{equation}
Now, choosing a Lipschitz approximation of a cylinder set $A_N$ (specified by conditions on coordinates $\omega_{{\bf i}_l}, l=1,..,N$), by Theorem \ref{thm3.1} we arrive at
 \begin{equation}\label{isoqin_q.1}
\mathcal{U}_q(\nu(A_N)) \le \tilde{c}\sqrt{N}\ \nu^+_2(A_N)
\end{equation}
with suitable constant $\tilde c\in(0,\infty)$ independent of $N$, and
with use of the subscript $2$ on the right-hand side to emphasise that we have here the surface measure with respect to the quadratic distance.
On the other hand using part {\rm (ii)} of Theorem \ref{betaLS+IFI_2for Gibbs}, we obtain
 \begin{equation}\label{isoqin_q.2}
\mathcal{U}_2(\nu(A_N)) \le \sqrt{C_2}\ \nu^+_2(A_N)
\end{equation}
Thus we obtain a potentially useful tool for optimisation of isoperimetric relations for finite dimensional marginals of the measure $\nu$.
\end{rem}

\section{Appendix }\label{Apx.I}
Suppose for  $d\mu \equiv e^{-U}d\lambda/Z$, with $U\geq \varepsilon$, for some $\varepsilon > 0$, and $Z$ a normalisation constant, we have
\[ \mu \left(f U^\beta\right) \leq C\mu |\nabla f| + D \mu f.
\]
In particular, for a Lipschitz cut-off function $0<\varepsilon\leq U_L\leq U$, for $f\equiv e^{\lambda U_L} U_L^\alpha$, with $\alpha, \beta>0,\ \alpha+\beta=1$, we have
\begin{align*}
\mu  \left(e^{\lambda U_L} U_L\right) = \mu   \left(e^{\lambda U_L} U_L^\alpha \cdot U_L^\beta\right)   &\leq C\mu |\nabla \left( e^{\lambda U_L} U_L^\alpha\right)| + D \mu \left( e^{\lambda U_L} U_L^\alpha\right) \\
&\leq 
\lambda C\mu \left( e^{\lambda U_L} U_L^\alpha\cdot |\nabla U_L|\right)
+ \alpha C\mu \left( e^{\lambda U_L} U_L^{ \alpha-1}\cdot |\nabla U_L|\right) \\ 
&+  D \mu \left( e^{\lambda U_L} U_L^\alpha\right).
\end{align*}
If we assume that
\[
|\nabla U_L| \leq a U_L^\beta 
\]
with $a \in(0,\infty)$ independent of $L$, then we get
\begin{align*}
 \mu \left( e^{\lambda U_L} U_L\right)   
&\leq 
\lambda C\mu \left( e^{\lambda U_L} U_L^\alpha\cdot (a U_L^\beta)\right) \\
&\qquad
+ \alpha C\mu \left( e^{\lambda U_L} U_L^{ \alpha-1}\cdot (a U_L^\beta )\right) +  D \mu \left( e^{\lambda U_L} U_L^\alpha\right)\\
&\leq \lambda a C\mu \left( e^{\lambda U_L} U_L \right)  
+ \alpha a C\mu \left( e^{\lambda U_L} \right) 
+ D\mu \left( e^{\lambda U_L} U_L^\alpha\right).
\end{align*}
Using our assumption that $U_L\geq\varepsilon>0$ and a bound
\[ U_L^\alpha \leq \lambda\delta U_L + A(\lambda\delta)
\]
with some $\delta, A(\lambda\delta)\in(0,\infty)$ independent of $L$, we get
\begin{align*}
 \mu  \left(e^{\lambda U_L} U_L\right)
&\leq 
 \lambda( a C+ D\delta )\mu \left( e^{\lambda U_L} U_L \right)   
+ 
\left(\alpha a C   +    D\cdot A(\lambda\delta)\right) \mu \left( e^{\lambda U_L} \right).
\end{align*}
Hence for $\lambda\in(0,\lambda_0 )$, with $\lambda_0\equiv( a C+ D\delta )^{-1}$, we have
\[ \frac{d}{d\lambda}\mu  \left(e^{\lambda U_L}\right) = \mu  \left(e^{\lambda U_L} U_L\right)
 \leq 
 B\mu \left( e^{\lambda U_L} \right)
\]
with 
$$B\equiv B(\lambda_0, \delta)\equiv \left(\alpha a C   +    D\cdot A(\lambda\delta)\right)\left(1-\lambda_0( a C+ D\delta )\right)^{-1}.$$
Solving this differential inequality for $\lambda\in(0,\lambda_0 )$, we obtain
\[\mu  \left(e^{\lambda U_L}\right)\leq  e^{\lambda B}.
\]
Since the constant $B$ is independent of $L$, by the dominated convergence theorem 
we obtain the following bound
\[\mu  \left(e^{\lambda U }\right)\leq  e^{\lambda B}
\]
true for $\lambda\in(0,\lambda_0 )$.\\

\bibliographystyle{siam}

\end{document}